\documentclass[11pt,reqno]{amsart}

\usepackage[T1]{fontenc}
\usepackage[utf8]{inputenc}
\usepackage{lmodern}
\usepackage[british]{babel}
\usepackage{mathtools,amssymb,amsthm}
\usepackage{array}
\usepackage{microtype}
\usepackage[hidelinks]{hyperref}
\usepackage{comment}
\hypersetup{
  pdftitle={Weighted well-posedness and kernel stability for coercive evolution equations with measure-valued delays},
  pdfauthor={Hiroki Ishizaka},
  pdfkeywords={measure-valued delay, coercive evolution equations, weighted well-posedness, total-variation stability, narrow convergence, delayed diffusion}
}
\usepackage[nameinlink,capitalize,noabbrev]{cleveref}

\usepackage[pagewise]{lineno}

\numberwithin{equation}{section}

\theoremstyle{plain}
\newtheorem{theorem}{Theorem}[section]
\newtheorem{lemma}[theorem]{Lemma}
\newtheorem{proposition}[theorem]{Proposition}
\newtheorem{corollary}[theorem]{Corollary}

\theoremstyle{definition}
\newtheorem{definition}[theorem]{Definition}

\theoremstyle{remark}
\newtheorem{remark}[theorem]{Remark}
\newtheorem{example}[theorem]{Example}

\crefname{theorem}{theorem}{theorems}
\Crefname{theorem}{Theorem}{Theorems}
\crefname{lemma}{lemma}{lemmas}
\Crefname{lemma}{Lemma}{Lemmas}
\crefname{proposition}{proposition}{propositions}
\Crefname{proposition}{Proposition}{Propositions}
\crefname{corollary}{corollary}{corollaries}
\Crefname{corollary}{Corollary}{Corollaries}
\crefname{definition}{definition}{definitions}
\Crefname{definition}{Definition}{Definitions}
\crefname{assumption}{assumption}{assumptions}
\Crefname{assumption}{Assumption}{Assumptions}
\crefname{remark}{remark}{remarks}
\Crefname{remark}{Remark}{Remarks}
\crefname{example}{example}{examples}
\Crefname{example}{Example}{Examples}

\newcommand{\R}{\mathbb{R}}
\newcommand{\M}{\mathcal{M}}
\newcommand{\W}{\mathcal{W}}
\newcommand{\K}{\mathcal{K}}
\newcommand{\E}{\mathsf{E}}
\newcommand{\dd}{\,\mathrm{d}}
\newcommand{\Tend}{\mathfrak{T}}
\newcommand{\norm}[1]{\left\lVert #1 \right\rVert}
\newcommand{\abs}[1]{\left\lvert #1 \right\rvert}
\newcommand{\pair}[2]{\left\langle #1,#2 \right\rangle}

\title[Weighted well-posedness and kernel stability]
{Weighted well-posedness and kernel stability for coercive evolution equations with measure-valued delays}
\thanks{Affiliation: Team FEM, Matsuyama, Japan.
Email: \texttt{h.ishizaka005@gmail.com}.}
\thanks{Earlier versions of this work appeared as
arXiv:2602.19099v1 (February 2026) and v2 (April 2026).
The present article is a corrected and substantially revised version,
with a new weighted well-posedness argument and a representative-independent
kernel-stability analysis at the natural energy regularity.}

\author{Hiroki Ishizaka}


\begin{document}
\raggedbottom

\begin{abstract}
We consider coercive evolution equations with measure-valued delay on a Gelfand triple.  The delayed feedback is induced by a bounded form on $V$ and may contain principal spatial derivatives, so it is naturally $V^*$-valued rather than bounded on the pivot space.  For every finite signed Borel kernel with no atom at the origin, an exponential weight makes the causal history operator contractive relative to the coercive parabolic solution operator.  This yields finite-time well-posedness without a smallness condition on the total variation and without any positivity assumption on the delayed form.  An asymmetric residual identity gives a total-variation Lipschitz estimate for signed kernels and strong stability for non-negative retarded kernels under narrow, equivalently weak, convergence of finite measures.  A delayed-diffusion realisation shows that the framework genuinely covers principal-part delays, and narrow stability yields qualitative distributed-to-discrete convergence at every fixed positive lag.
\end{abstract}

\maketitle

\section{Introduction}\label{sec:introduction}
Let $\Tend>0$ be a final time and let $\tau>0$ be a maximal delay.  We study
\begin{subequations}\label{eq:intro-history}
\begin{align}
\partial_tu+A_0u+\int_{[0,\tau]}A_1\widetilde u(t-s)\,\dd\mu(s)&=f,
\quad 0<t<\Tend,\label{eq:intro-history=a}\\
u(0)&=u_0,\label{eq:intro-history=b}
\end{align}
\end{subequations}
where $\widetilde u$ joins the unknown trajectory on $(0,\Tend)$ to a prescribed past on $(-\tau,0)$ and the finite signed Borel measure $\mu$ specifies the delay law.  This formulation includes distributed, singular continuous, and atomic delays.

The main analytical difficulty is that the delayed feedback need not be bounded on the pivot space.  We work on a Gelfand triple $V\hookrightarrow H\cong H^*\hookrightarrow V^*$ and assume that $A_0,A_1:V\to V^*$ are generated by bounded bilinear forms, with the instantaneous form coercive.  Thus $A_1$ may contain the same spatial derivatives as the principal diffusion operator.  Principal-part memory of this kind is classical in heat conduction with memory and in viscoelasticity, where the constitutive flux depends on the history of the temperature gradient or the strain~\cite{GurtinPipkin1968,Miller1978,Dafermos1970}.  In the delayed-diffusion model, the variational Laplacian acts from $H_0^1(\Omega)$ to $H^{-1}(\Omega)$ and has no bounded extension from $L^2(\Omega)$ to $H^{-1}(\Omega)$.  The prescribed history is therefore taken only in $L^2(-\tau,0;V)$, and the measure-valued history term must be defined without assigning point values to an $L^2$-equivalence class.

Abstract delay semigroups, Volterra equations, and exponential weights in causal evolution equations are classical \cite{Gripenberg1990,Pruss1993,HaleVerduynLunel1993,Wu1996,BatkaiPiazzera2005,TravisWebb1974}.  Continuous dependence on delay parameters for parabolic equations is studied in \cite{KryspinMierczynski2023,KryspinMierczynski2024}, while measure-valued delay kernels arise also in parabolic control \cite{CasasMateosTroeltzsch2018}.

Earlier versions of the present work, arXiv:2602.19099v1 (February 2026) and v2 (April 2026), introduced the measure-valued formulation and contained preliminary finite-time and kernel-continuity arguments.  The present version corrects the former a priori analysis and replaces it by the weighted causal estimate developed below.  It also gives a representative-independent formulation for atomic kernels with square-integrable histories and establishes residual, total-variation, and narrow kernel stability at the natural energy regularity.

Subsequently, Shikhman \cite{Shikhman2026} studied a semilinear reaction--diffusion equation with finite signed delay measures in the continuous-history phase space $\mathcal C([-\tau,0];L^2(\Omega))$, including well-posedness, kernel continuity, and long-time dynamics.  The analysis in \cite{Shikhman2026} is based on the earlier arXiv version as it stood at the time.  The two approaches address complementary functional-analytic regimes: Shikhman's setting uses continuous pivot-space histories, whereas the present article treats form-valued delayed operators $A_1:V\to V^*$ and prescribed histories in $L^2(-\tau,0;V)$.

The contributions are the following.  The measure-valued history operator is defined intrinsically on square-integrable histories, so that atomic delays act on $L^2$-equivalence classes without point evaluation.  In the energy-space regime the delayed operator admits no bounded extension to the pivot space (\Cref{prop:not-H-bounded}), so histories must be taken in $L^2(-\tau,0;V)$; this is what separates form-valued, principal-part delays from the continuous pivot-space histories appropriate to reaction delays.  An asymmetric residual identity then gives a total-variation Lipschitz estimate for signed measures and strong stability under narrow convergence for non-negative kernels, with no sign condition on the delayed form $a_1$ and no smallness condition on the kernel mass.  The weighted well-posedness theory of \Cref{sec:weighted} is the foundational vehicle for these results rather than an end in itself.

The well-posedness is obtained from an exponentially weighted norm and the associated variation mass.  For a real interval $I$ and a Hilbert space $X$, we set
\begin{align}
  \norm{g}_{L^2_\beta(I;X)}
  &:=\left(\int_I e^{-2\beta t}\norm{g(t)}_X^2\,\dd t\right)^{1/2},
  \label{eq:intro-weighted-norm}\\
  \kappa_\mu(\beta)
  &:=\int_{[0,\tau]}e^{-\beta s}\,\dd|\mu|(s),
  \quad \beta\ge0.\label{eq:kappa}
\end{align}
If $|\mu|(\{0\})=0$, then $\kappa_\mu(\beta)\to0$ as $\beta\to\infty$.  The causal history operator consequently becomes a strict perturbation of the coercive parabolic solution operator in a suitable exponentially weighted norm.  This proves finite-time well-posedness for any fixed signed retarded kernel without requiring $\|\mu\|_{\M([0,\tau])}$ to be small.

The stability theory is built on an asymmetric residual identity.  For solutions with common data, the kernel on the left determines the weighted resolvent margin, while the kernel on the right supplies the reference trajectory along which the perturbation is evaluated.  This yields an asymmetric Lipschitz estimate in total variation for signed kernels.  For non-negative kernels, translation continuity of a fixed joined trajectory also gives strong solution convergence under narrow convergence of measures. Because absolutely continuous approximations of an atom remain at total-variation distance twice their mass, the narrow topology is the natural qualitative topology for distributed-to-discrete concentration.

Robustness of the solution with respect to the delay law is, moreover, the analytic prerequisite for two control-theoretic questions to which the present linear theory is a natural first step: the identification of a memory kernel from observations of the state, and the analysis of feedback acting through distributed or discrete delay.  Neither problem is treated here, but the total-variation and narrow stability estimates below are of the kind such questions require.

The paper is organised as follows.  \Cref{sec:setting,sec:weighted} develop the measure-valued history operator and weighted well-posedness.  \Cref{sec:pde-realisation} treats delayed diffusion fluxes and explains the failure of pivot-space boundedness.  \Cref{sec:stability} proves residual, total-variation, and narrow kernel stability and concludes with qualitative distributed-to-discrete convergence at a fixed positive lag.

\section{Functional setting and measure-valued histories}\label{sec:setting}
Let $V$ and $H$ be real separable Hilbert spaces such that $V\hookrightarrow H$ continuously and densely.  Identifying $H$ with its dual gives the Gelfand triple
\begin{align*}
V\hookrightarrow H\cong H^*\hookrightarrow V^*.
\end{align*}
The duality pairing between $V^*$ and $V$ is denoted by $\pair{\cdot}{\cdot}$, and $\mathcal L(X,Y)$ denotes the space of bounded linear operators from a Banach space $X$ to a Banach space $Y$.

For an interval $I\subset\R$, a Banach space $X$, and $1 \leq p \leq \infty$, the Bochner spaces are denoted by $L^p(I;X)$.  If $g\in L^2(I;X)$, its equivalence class is not assigned point values unless a continuous representative is known.  This elementary point matters for the past history: the condition $\psi(0)=u_0$ is not meaningful for a general $L^2(-\tau,0;V)$ function.

Let $\M([0,\tau])$ be the Banach space of finite signed Borel measures on $[0,\tau]$, with total-variation norm
\begin{align*}
  \norm{\sigma}_{\M([0,\tau])}:=|\sigma|([0,\tau]).
\end{align*}
Its positive cone is $\M_+([0,\tau])$.  We also set
\begin{align*}
  \M^0([0,\tau])
  &:=\{\sigma\in\M([0,\tau]):|\sigma|(\{0\})=0\},\\
  \M_+^0([0,\tau])
  &:=\M^0([0,\tau])\cap\M_+([0,\tau]).
\end{align*}
For any $\sigma\in\M([0,\tau])$, let $\sigma=\sigma^+-\sigma^-$ be its Jordan decomposition, so that $|\sigma|=\sigma^++\sigma^-$. If $X$ is a Banach space and $F:[0,\tau]\to X$ is strongly $|\sigma|$-measurable with
\begin{align*}
\int_{[0,\tau]}\norm{F(s)}_X\,\dd|\sigma|(s)<\infty,
\end{align*}
we define
\begin{align*}
  \int_{[0,\tau]}F\,\dd\sigma
  :=
  \int_{[0,\tau]}F\,\dd\sigma^+
  -
  \int_{[0,\tau]}F\,\dd\sigma^-,
\end{align*}
where the two integrals on the right are Bochner integrals. The well-posedness and total-variation stability results allow arbitrary signed retarded kernels $\sigma\in\M^0([0,\tau])$. Non-negativity is required only in the section on narrow convergence.

On the compact interval $[0,\tau]$, narrow convergence of non-negative measures means
\begin{align*}
\mu_n\rightharpoonup\mu
\quad \Longleftrightarrow \quad
\int\varphi\,\dd\mu_n\longrightarrow\int\varphi\,\dd\mu \quad\text{for any }\varphi\in \mathcal{C}([0,\tau]).
\end{align*}
It implies convergence, and hence uniform boundedness, of the total masses by testing with $\varphi\equiv1$.

Let $a_0,a_1:V\times V\to\R$ be bilinear forms.  We assume that there exist constants $\Lambda_0,\Lambda_1,\alpha_0>0$ such that
\begin{align}
\abs{a_i(w,v)} &\leq \Lambda_i\norm{w}_V\norm{v}_V, \quad w,v\in V,\quad i\in\{0,1\}, \label{eq:form-bounded}\\
a_0(v,v) &\geq \alpha_0\norm{v}_V^2, \quad v\in V. \label{eq:a0-coercive}
\end{align}
No symmetry or sign condition is imposed on $a_1$.  Let $A_i\in\mathcal L(V,V^*)$ be defined by
\begin{align*}
  \pair{A_iw}{v}=a_i(w,v).
\end{align*}

\begin{example}[Diffusion realisation]\label[example]{ex:diffusion}
Let $\Omega\subset\R^d$, $d\in\{1,2,3\}$, be a bounded Lipschitz domain.  We set
$H=L^2(\Omega)$ and $V=H_0^1(\Omega)$.  Let
$B_0,B_1\in L^\infty(\Omega;\R^{d\times d})$, and assume that there exists $\lambda_0>0$ such that
\begin{align*}
  \xi^\top B_0(x)\xi\geq\lambda_0\abs{\xi}^2
  \quad\text{for every }\xi\in\R^d
  \text{ and for almost every }x\in\Omega.
\end{align*}
Then,
\begin{align*}
a_i(w,v)=\int_\Omega B_i(x)\nabla w\cdot\nabla v\,\dd x
\end{align*}
satisfies \eqref{eq:form-bounded}--\eqref{eq:a0-coercive}.  In general, $A_1:V\to V^*$ does not admit a bounded extension from $H$ into $V^*$; this is precisely the energy-space regime considered below.
\end{example}

\begin{example}[Kernel classes]\label[example]{ex:kernel-classes}
The following kernels all belong to the measure framework.  For $r\in[0,\tau]$, the symbol $\delta_r$ denotes the unit Dirac measure at $r$.
\begin{enumerate}
\item An absolutely continuous distributed memory has $\dd\mu(s)=k(s)\,\dd s$ with $k\in L^1(0,\tau)$, $k\ge0$.
\item A finite multiple delay is $\mu=\sum_{j=1}^Jm_j\delta_{\tau_j}$, where $m_j\in\R\setminus\{0\}$ and $0<\tau_j\le\tau$.
\item A countable atomic delay is $\mu=\sum_{j\ge1}m_j\delta_{\tau_j}$ with $\sum_j|m_j|<\infty$.  Accumulation of the delay locations inside $(0,\tau]$ is allowed.
\item A mixed kernel is the sum of an absolutely continuous component, a singular continuous component, and an atomic component.  No decomposition of this kind is used in the proofs.
\end{enumerate}
The memory-free problem corresponds to $\mu=0$.
\end{example}

Recall that $\tau>0$ is a delay horizon and $\Tend>0$ is a final time. The condition $|\mu|(\{0\})=0$ isolates genuinely retarded terms.  An atom at zero is instantaneous and may instead be incorporated into $a_0$; see \Cref{rem:atom-zero}.

The following lemma makes the measure-valued history operator precise before it is used in the weak formulation.  In particular, it treats atomic kernels without assigning point values to an $L^2$-equivalence class.

\begin{lemma}[Well-definedness of the measure-valued history operator]
\label[lemma]{lem:measure-convolution}
Let $\sigma\in\M([0,\tau])$ and $g\in L^2(-\tau,\Tend;V)$.  Choose a Borel strongly measurable representative of $g$ and extend it by zero to $\R$.  Then, the Bochner integral
\begin{align}\label{eq:full-history-operator}
  (\K_\sigma g)(t)
  :=\int_{[0,\tau]}A_1g(t-s)\,\dd\sigma(s)
\end{align}
exists in $V^*$ for almost every $t\in(0,\Tend)$ and defines an element of the space $L^2(0,\Tend;V^*)$.  Furthermore,
\begin{align}\label{eq:measure-convolution-bound}
  \norm{\K_\sigma g}_{L^2(0,\Tend;V^*)}
  \leq
  \Lambda_1\norm{\sigma}_{\M([0,\tau])}
  \norm{g}_{L^2(-\tau,\Tend;V)}.
\end{align}
The element $\K_\sigma g\in L^2(0,\Tend;V^*)$ is independent of the chosen Borel strongly measurable representative of $g$. More generally, it is unchanged under any modification on a Lebesgue-null subset of $\mathbb R$, provided that the modified function remains Borel strongly measurable.
\end{lemma}

\begin{proof}
Because $V$ is separable, the element $g\in L^2(-\tau,\Tend;V)$ admits a Borel measurable representative, which is strongly measurable.  The map $(t,s)\mapsto g(t-s)$ is Borel measurable. Because $A_1\in\mathcal L(V,V^*)$, the map
\begin{align*}
  (t,s)\longmapsto A_1g(t-s)
\end{align*}
is Borel measurable and takes its values in the separable subspace $\overline{A_1(V)}\subset V^*$.  It is therefore strongly measurable on $(0,\Tend)\times[0,\tau]$. Furthermore, by the boundedness of $A_1$, the Cauchy--Schwarz inequality in time, and the finiteness of $|\sigma|$,
\begin{align*}
  &\int_{[0,\tau]}\int_0^{\Tend}
  \norm{A_1g(t-s)}_{V^*}\,\dd t\,\dd|\sigma|(s)\\
  &\quad\leq
  \Lambda_1\Tend^{1/2}
  \int_{[0,\tau]}
  \left(\int_0^{\Tend}\norm{g(t-s)}_V^2\,\dd t\right)^{1/2}
  \dd|\sigma|(s)\\
  &\quad\leq
  \Lambda_1\Tend^{1/2}
  \norm{\sigma}_{\M([0,\tau])}
  \norm{g}_{L^2(-\tau,\Tend;V)}<\infty.
\end{align*}
The Bochner--Fubini theorem therefore shows that the integral in \eqref{eq:full-history-operator} exists for almost every $t$. After setting $(\K_\sigma g)(t):=0$ on the exceptional Lebesgue-null set, this defines a strongly measurable $V^*$-valued function.

For almost every $t$, the variation estimate for the Bochner integral gives
\begin{align*}
  \norm{(\K_\sigma g)(t)}_{V^*}
  \leq
  \Lambda_1\int_{[0,\tau]}\norm{g(t-s)}_V\,\dd|\sigma|(s).
\end{align*}
Minkowski's integral inequality and the zero extension of $g$ yield
\begin{align*}
  \norm{\K_\sigma g}_{L^2(0,\Tend;V^*)}
  &\leq
  \Lambda_1\int_{[0,\tau]}
  \norm{g(\cdot-s)}_{L^2(0,\Tend;V)}\,\dd|\sigma|(s)\\
  &\leq
  \Lambda_1\norm{\sigma}_{\M([0,\tau])}
  \norm{g}_{L^2(-\tau,\Tend;V)},
\end{align*}
which proves \eqref{eq:measure-convolution-bound}.

It remains to prove independence of the representative and of the values assigned on Lebesgue-null sets.  Let $g_1,g_2:\mathbb R\to V$ be Borel strongly measurable functions
such that
\begin{align*}
  g_1=g_2
  \quad\text{Lebesgue-almost everywhere on }\mathbb R.
\end{align*}
This includes two Borel strongly measurable representatives of the same element of $L^2(-\tau,\Tend;V)$, extended by zero outside $(-\tau,\Tend)$, as well as any Borel strongly measurable modifications of these extensions on a Lebesgue-null subset of $\mathbb R$.  We set
\begin{align*}
  N:=\{r\in\mathbb R:
  \norm{g_1(r)-g_2(r)}_V>0\}.
\end{align*}
Then, $N$ is Borel and has Lebesgue measure zero.  The set
\begin{align*}
  E:=\{(t,s)\in(0,\Tend)\times[0,\tau]:t-s\in N\}
\end{align*}
is Borel.  For each fixed $s\in[0,\tau]$, its $t$-section is $(N+s)\cap(0,\Tend)$ and hence has Lebesgue measure zero. Tonelli's theorem therefore gives
\begin{align*}
0
&=\int_{[0,\tau]}\int_0^{\Tend} \mathbf 1_E(t,s)\,\dd t\,\dd|\sigma|(s) =\int_0^{\Tend}\int_{[0,\tau]} \mathbf 1_E(t,s)\,\dd|\sigma|(s)\,\dd t.
\end{align*}
Consequently, for almost every $t\in(0,\Tend)$ one has $g_1(t-s)=g_2(t-s)$ for $|\sigma|$-almost every $s\in[0,\tau]$.  Therefore,
\begin{align*}
  \int_{[0,\tau]}A_1g_1(t-s)\,\dd\sigma(s)
  =
  \int_{[0,\tau]}A_1g_2(t-s)\,\dd\sigma(s)
\end{align*}
for almost every $t\in(0,\Tend)$.  This proves the final assertion.
\end{proof}

Thus, for any $\sigma\in\M([0,\tau])$, the notation $\K_\sigma$ denotes the bounded linear operator
\begin{align*}
  \K_\sigma:
  L^2(-\tau,\Tend;V)
  &\longrightarrow L^2(0,\Tend;V^*),\\
  g&\longmapsto
  \left[
    t\longmapsto
    \int_{[0,\tau]}A_1g(t-s)\,\dd\sigma(s)
  \right],
\end{align*}
where $g$ is extended by zero outside $(-\tau,\Tend)$. Whenever $G\in L^2(\R;V)$, we use the same notation and set
\begin{align*}
  \K_\sigma G
  :=\K_\sigma\bigl(G|_{(-\tau,\Tend)}\bigr).
\end{align*}
Equivalently,
\begin{align*}
  (\K_\sigma G)(t)
  =
  \int_{[0,\tau]}A_1G(t-s)\,\dd\sigma(s)
\end{align*}
for almost every $t\in(0,\Tend)$.

We prescribe
\begin{align*}
  u_0\in H,\quad f\in L^2(0,\Tend;V^*),\quad
  \psi\in L^2(-\tau,0;V).
\end{align*}
At this regularity, $\psi(0)$ is not defined, and no compatibility condition at the single time $t=0$ is required for the basic well-posedness theory.  If $\psi$ admits a representative in $\mathcal{C}([-\tau,0];H)$, one may additionally impose the compatibility condition
\begin{align*}
  \psi(0)=u_0\quad\text{in }H.
\end{align*}
Such compatibility will be imposed later only when it is required by additional time regularity.

For a function $u:(0,\Tend)\to V$, we define the joined trajectory
\begin{align}\label{eq:joined-trajectory}
  \tilde u(t):=
  \begin{cases}
    \psi(t),&-\tau<t<0,\\
    u(t),&0<t<\Tend.
  \end{cases}
\end{align}
Values assigned to this auxiliary trajectory at the endpoints are immaterial for its $L^2$-equivalence class.

\begin{lemma}[Joined trajectory]\label[lemma]{lem:joined-trajectory}
Let $u\in L^2(0,\Tend;V)$ and $\psi\in L^2(-\tau,0;V)$.  The function $\tilde u$ defined by \eqref{eq:joined-trajectory} belongs to $L^2(-\tau,\Tend;V)$, and
\begin{align*}
  \norm{\tilde u}_{L^2(-\tau,\Tend;V)}^2
  =\norm{\psi}_{L^2(-\tau,0;V)}^2
   +\norm{u}_{L^2(0,\Tend;V)}^2.
\end{align*}
Changing the values assigned to the auxiliary joined trajectory at $-\tau$, $0$, or $\Tend$ does not change $\K_\sigma\tilde u$ in $L^2(0,\Tend;V^*)$ for any $\sigma\in\M([0,\tau])$.
\end{lemma}

\begin{proof}
We choose Borel strongly measurable representatives of $\psi$ and $u$. Assign an arbitrary value to $\tilde u(0)$ and define
\begin{align*}
\tilde u(t):=
\begin{cases}
\psi(t), & -\tau<t<0,\\
u(t),    & 0<t<\Tend.
\end{cases}
\end{align*}
Then, $\tilde u:(-\tau,\Tend)\to V$ is strongly measurable. Furthermore,
\begin{align*}
  \int_{-\tau}^{\Tend}\norm{\tilde u(t)}_V^2\,\dd t
  &=
  \int_{-\tau}^{0}\norm{\psi(t)}_V^2\,\dd t
  +
  \int_{0}^{\Tend}\norm{u(t)}_V^2\,\dd t\\
  &=
  \norm{\psi}_{L^2(-\tau,0;V)}^2
  +
  \norm{u}_{L^2(0,\Tend;V)}^2.
\end{align*}
Therefore, $\tilde u\in L^2(-\tau,\Tend;V)$ and the asserted norm identity holds. Changing the values assigned at $-\tau$, $0$, or $\Tend$ modifies the chosen representative or its extension only on a finite, and hence Lebesgue-null, subset of $\mathbb R$. Therefore, by \Cref{lem:measure-convolution}, such changes do not alter $\K_\sigma\tilde u$ as an element of $L^2(0,\Tend;V^*)$.
\end{proof}

We set
\begin{align*}
\W(0,\Tend)
  &:=
  \left\{
    u\in L^2(0,\Tend;V):
    \partial_tu\in L^2(0,\Tend;V^*)
  \right\},
\end{align*}
equipped with
\begin{align*}
\norm{u}_{\W(0,\Tend)}
  &:=
  \left(
    \norm{u}_{L^2(0,\Tend;V)}^2
    +
    \norm{\partial_tu}_{L^2(0,\Tend;V^*)}^2
  \right)^{1/2}.
\end{align*}
By the Lions--Magenes lemma \cite{LionsMagenes1972}, $\W(0,\Tend)\hookrightarrow \mathcal{C}([0,\Tend];H)$.

We consider
\begin{align}\label{eq:problem}
\partial_tu+A_0u+\K_\mu\tilde u=f \quad\text{in }V^*\text{ for a.e. }t\in(0,\Tend), \quad u(0)=u_0\quad\text{in }H.
\end{align}

\begin{definition}[Weak solution]\label[definition]{def:weak-solution}
Let $\mu\in\M^0([0,\tau])$.  A function $u\in\W(0,\Tend)$ is a weak solution of \eqref{eq:problem} if $u(0)=u_0$ in $H$ and
\begin{align*}
  \partial_tu(t)+A_0u(t)+\K_\mu\tilde u(t)=f(t)
  \quad\text{in }V^*
\end{align*}
for almost every $t\in(0,\Tend)$.  Equivalently, outside one Lebesgue-null set of times,
\begin{align*}
  \pair{\partial_tu(t)}{v}+a_0(u(t),v)
  +\int_{[0,\tau]}a_1(\tilde u(t-s),v)\,\dd\mu(s)
  =\pair{f(t)}{v}
\end{align*}
for any $v\in V$.  The values assigned to the auxiliary joined representative at individual times do not affect the history term as an element of $L^2(0,\Tend;V^*)$.  These arbitrary pointwise assignments are unrelated to the initial trace of $u$: if $u\in\W(0,\Tend)\hookrightarrow \mathcal{C}([0,\Tend];H)$, then $u(0)=u_0$ refers to the unique $H$-continuous representative of $u$.
\end{definition}

\begin{remark}[An atom at the origin]\label[remark]{rem:atom-zero}
Let $\hat\mu$ be a finite signed measure on $[0,\tau]$.  We set
\begin{align*}
  m_0:=\hat\mu(\{0\}),
  \qquad
  \mu:=\hat\mu-m_0\delta_0.
\end{align*}
Then,
\begin{align*}
  \hat\mu=m_0\delta_0+\mu,
  \qquad
  |\mu|(\{0\})=0.
\end{align*}
For almost every $t\in(0,\Tend)$ and any $v\in V$,
\begin{align*}
  \left\langle\K_{\hat\mu}\tilde u(t),v\right\rangle
  =
  m_0a_1\bigl(u(t),v\bigr)
  +
  \left\langle\K_\mu\tilde u(t),v\right\rangle.
\end{align*}
Thus, the atom at the origin may be absorbed into the
instantaneous form by setting
\begin{align*}
  \hat a_0:=a_0+m_0a_1.
\end{align*}
The same well-posedness theory applies provided that there exists $\hat\alpha_0>0$ such that
\begin{align*}
  \hat a_0(v,v)
  =
  a_0(v,v)+m_0a_1(v,v)
  \geq
  \hat\alpha_0\norm{v}_V^2
  \quad\text{for all }v\in V.
\end{align*}
Under the standing coercivity assumption on $a_0$, the condition
\begin{align*}
  m_0a_1(v,v)\geq0
  \quad\text{for all }v\in V
\end{align*}
is sufficient, but not necessary.  In particular, if $m_0\geq0$, the non-negativity of $a_1$ is sufficient.  We therefore formulate the genuinely retarded part of the well-posedness theory for $\M^0([0,\tau])$; positivity is imposed later only where it is needed for narrow convergence.
\end{remark}

\section{Weighted history estimates and well-posedness}
\label{sec:weighted}
This section establishes the weighted history estimates used to prove finite-time well-posedness.  After separating the causal history of the unknown trajectory from the contribution of the prescribed prehistory, we show that the causal operator has weighted norm at most $\Lambda_1\kappa_\mu(\beta)$.  For a genuinely retarded measure, $\kappa_\mu(\beta)\to0$ as $\beta\to\infty$, which permits the history term to be absorbed into the coercive instantaneous estimate and removes any smallness requirement on the total variation of $\mu$.

Unless otherwise stated, throughout this section
\begin{align*}
  \mu\in\M^0([0,\tau]).
\end{align*}

\subsection{The weighted norm and the weighted variation mass}
We use the weighted norm and the quantity $\kappa_\mu(\beta)$ defined in \eqref{eq:intro-weighted-norm} and \eqref{eq:kappa}, respectively. Because $|\mu|(\{0\})=0$, dominated convergence gives 
\begin{align}\label{eq:kappa-decay}
\kappa_\mu(\beta)\longrightarrow 0 \quad \text{as } \beta\to\infty.
\end{align}

\begin{lemma}[Elementary properties of the weighted norm]
\label[lemma]{lem:weighted-equivalence}
For any $\beta\geq0$, any Hilbert space $X$, and any strongly measurable function $g:(0,\Tend)\to X$,
\begin{align}\label{eq:weighted-equivalence}
  e^{-\beta\Tend}\norm{g}_{L^2(0,\Tend;X)}
  \leq
  \norm{g}_{L^2_\beta(0,\Tend;X)}
  \leq
  \norm{g}_{L^2(0,\Tend;X)},
\end{align}
where the inequalities are understood in the extended non-negative reals.  Consequently,
\begin{align*}
  L^2_\beta(0,\Tend;X)=L^2(0,\Tend;X)
\end{align*}
as sets, and the two norms are equivalent on any fixed finite interval.
\end{lemma}

\begin{proof}
For any $t\in(0,\Tend)$,
\begin{align*}
  e^{-2\beta\Tend}
  \leq e^{-2\beta t}\leq1.
\end{align*}
Multiplying by $\norm{g(t)}_X^2$ and integrating over
$(0,\Tend)$ gives
\begin{align*}
  e^{-2\beta\Tend}
  \norm{g}_{L^2(0,\Tend;X)}^2
  \leq
  \norm{g}_{L^2_\beta(0,\Tend;X)}^2
  \leq
  \norm{g}_{L^2(0,\Tend;X)}^2.
\end{align*}
Taking square roots proves \eqref{eq:weighted-equivalence}.

Conversely, if $h\in L^2_\beta(0,\Tend;X)$, then
\begin{align*}
  \norm{h}_{L^2(0,\Tend;X)}
  \leq
  e^{\beta\Tend}
  \norm{h}_{L^2_\beta(0,\Tend;X)},
\end{align*}
so $h\in L^2(0,\Tend;X)$.  Therefore, the weighted and unweighted spaces coincide as sets, and their norms are equivalent.
\end{proof}

\begin{lemma}[Weighted variation mass and the instantaneous atom]\label[lemma]{lem:kappa-properties}
Let $\hat \mu\in\M([0,\tau])$, and define $\kappa_{\hat \mu}(\beta) :=\int_{[0,\tau]}e^{-\beta s}\,\dd|\hat \mu|(s)$. Then, $\beta\mapsto\kappa_{\hat \mu}(\beta)$ is non-increasing and continuous on $[0,\infty)$, and
\begin{align}\label{eq:kappa-atom-limit}
  \lim_{\beta\to\infty}\kappa_{\hat \mu}(\beta)
  =|\hat \mu|(\{0\}).
\end{align}
Consequently, the weighted variation mass can be made arbitrarily small precisely when the signed measure has no atom at the origin.
\end{lemma}

\begin{proof}
For each $s\in[0,\tau]$, the function $\beta\mapsto e^{-\beta s}$ is non-increasing.  Therefore, $\beta\mapsto\kappa_{\hat\mu}(\beta)$ is non-increasing. 

To prove continuity, let $\beta_n\to\beta$ in $[0,\infty)$. Then,
\begin{align*}
  e^{-\beta_n s}\longrightarrow e^{-\beta s}
  \qquad\text{for every }s\in[0,\tau],
\end{align*}
and
\begin{align*}
  0\leq e^{-\beta_n s}\leq1.
\end{align*}
Because $|\hat\mu|$ is a finite positive measure, the dominated convergence theorem yields
\begin{align*}
  \kappa_{\hat\mu}(\beta_n)
  \longrightarrow
  \kappa_{\hat\mu}(\beta).
\end{align*}

Finally,
\begin{align*}
  e^{-\beta s}
  \longrightarrow
  \mathbf 1_{\{0\}}(s)
  \quad\text{as }\beta\to\infty.
\end{align*}
Another application of the dominated convergence theorem yields
\begin{align*}
  \lim_{\beta\to\infty}\kappa_{\hat\mu}(\beta)
  =
  \int_{[0,\tau]}\mathbf 1_{\{0\}}(s)\,\dd|\hat\mu|(s)
  =
  |\hat\mu|(\{0\}).
\end{align*}
Consequently,
\begin{align*}
  \inf_{\beta\geq0}\kappa_{\hat\mu}(\beta)
  =
  |\hat\mu|(\{0\}),
\end{align*}
so $\kappa_{\hat\mu}(\beta)$ can be made arbitrarily small if and only if $|\hat\mu|(\{0\})=0$, equivalently $\hat\mu(\{0\})=0$.
\end{proof}

\subsection{Measure convolution and causal splitting}
Let $\E_+w$ denote the zero extension of $w:(0,\Tend)\to V$ to $\R$, and let $\E_-\psi$ be the extension of $\psi:(-\tau,0)\to V$ by zero outside $(-\tau,0)$.  Because
\begin{align*}
\tilde u=\E_+u+\E_-\psi \quad \text{in }L^2(-\tau,\Tend;V),
\end{align*}
the bounded linearity of $\K_\mu$ established in \Cref{lem:measure-convolution} gives
\begin{align}\label{eq:history-splitting}
\K_\mu\tilde u
=\K_\mu^+u+h_{\mu,\psi} \quad\text{in }L^2(0,\Tend;V^*),
\end{align}
where
\begin{align*}
\K_\mu^+u:=\K_\mu(\E_+u), \quad h_{\mu,\psi}:=\K_\mu(\E_-\psi).
\end{align*}
Applying \eqref{eq:measure-convolution-bound} to $g=\E_-\psi$ and using $e^{-\beta t}\leq1$ on $(0,\Tend)$ yields
\begin{align}\label{eq:history-forcing-bound}
\norm{h_{\mu,\psi}}_{L^2_\beta(0,\Tend;V^*)}
\leq
\Lambda_1\norm{\mu}_{\M([0,\tau])} \norm{\psi}_{L^2(-\tau,0;V)}.
\end{align}

\subsection{The weighted causal estimate}
The following estimate is the key point of the paper.  The factor $e^{-\beta s}$ appears because the present-time weight $e^{-\beta t}$ is transported to the delayed time $t-s$.  Only the unknown causal part is estimated by the weighted variation mass. The prescribed history remains an external forcing.

\begin{lemma}[Weighted causal estimate]
\label[lemma]{lem:weighted-causal}
Let $\beta\geq0$.  For any $w\in L^2_\beta(0,\Tend;V)$,
\begin{align}\label{eq:weighted-causal}
\norm{\K_\mu^+w}_{L^2_\beta(0,\Tend;V^*)}
\leq
\Lambda_1\kappa_\mu(\beta) \norm{w}_{L^2_\beta(0,\Tend;V)}.
\end{align}
\end{lemma}

\begin{proof}
By \Cref{lem:weighted-equivalence},
$w\in L^2(0,\Tend;V)$.  Choose a Borel strongly measurable representative of $w$ and we set
\begin{align*}
w_\beta(t):=e^{-\beta t}w(t), \quad 0<t<\Tend.
\end{align*}
Extend $w_\beta$ by zero to $\mathbb R$.  Then,
\begin{align*}
  \norm{w_\beta}_{L^2(-\tau,\Tend;V)}
  &=
  \norm{w_\beta}_{L^2(0,\Tend;V)}
  =
  \norm{w}_{L^2_\beta(0,\Tend;V)}.
\end{align*}
We define a finite signed Borel measure $\mu_\beta$ on $[0,\tau]$ as
\begin{align*}
  \mu_\beta(B)
  :=
  \int_B e^{-\beta s}\,\dd\mu(s)
\end{align*}
for any Borel set $B\subset[0,\tau]$. Because $e^{-\beta s}\geq0$,
\begin{align*}
\dd|\mu_\beta|(s)
=
e^{-\beta s}\,\dd|\mu|(s),
\quad
\norm{\mu_\beta}_{\M([0,\tau])}
=
\kappa_\mu(\beta).
\end{align*}
For almost every $t\in(0,\Tend)$,
\begin{align*}
e^{-\beta t}(\K_\mu^+w)(t)
&=
\int_{[0,\tau]} e^{-\beta t}A_1w(t-s)\,\dd\mu(s)
=
\int_{[0,\tau]} e^{-\beta s}A_1w_\beta(t-s)\,\dd\mu(s)\\
&= (\K_{\mu_\beta}w_\beta)(t),
\end{align*}
where the zero extensions are used when $t-s\leq0$.  Therefore, by \Cref{lem:measure-convolution},
\begin{align*}
\norm{\K_\mu^+w}_{L^2_\beta(0,\Tend;V^*)}
&= \norm{\K_{\mu_\beta}w_\beta}_{L^2(0,\Tend;V^*)} \\
&\leq
\Lambda_1 \norm{\mu_\beta}_{\M([0,\tau])} \norm{w_\beta}_{L^2(-\tau,\Tend;V)}
= \Lambda_1\kappa_\mu(\beta) \norm{w}_{L^2_\beta(0,\Tend;V)},
\end{align*}
which proves \eqref{eq:weighted-causal}.
\end{proof}

\subsection{The coercive parabolic solution operator}
We introduce the precise weighted estimate for the unperturbed problem. The exponential change of unknown converts the weighted estimate into the standard energy estimate for a coercive parabolic equation with an additional non-negative pivot-space term.

\begin{lemma}[Weighted parabolic estimate]\label[lemma]{lem:weighted-parabolic}
Let $g\in L^2_\beta(0,\Tend;V^*)$ and $y_0\in H$.  The problem
\begin{align*}
\partial_ty+A_0y=g,\quad y(0)=y_0,
\end{align*}
has a unique solution $y\in\W(0,\Tend)$.  With
\begin{align*}
  \mathcal R_\beta(y_0,g)
  :=\norm{y_0}_H^2
    +\alpha_0^{-1}\norm{g}_{L^2_\beta(0,\Tend;V^*)}^2,
\end{align*}
one has the two separate estimates
\begin{align}
\sup_{0\le t\le \Tend}e^{-2\beta t}\norm{y(t)}_H^2
&\leq \mathcal R_\beta(y_0,g), \label{eq:weighted-parabolic-sup}\\
\alpha_0\norm{y}_{L^2_\beta(0,\Tend;V)}^2
&\leq \mathcal R_\beta(y_0,g). \label{eq:weighted-parabolic-L2}
\end{align}
Consequently,
\begin{align}\label{eq:weighted-parabolic-energy}
\sup_{0\le t\le \Tend}e^{-2\beta t}\norm{y(t)}_H^2 +\alpha_0\norm{y}_{L^2_\beta(0,\Tend;V)}^2 \leq 2\mathcal R_\beta(y_0,g).
\end{align}
If $y_0=0$, then
\begin{align}\label{eq:parabolic-Lipschitz}
\norm{y}_{L^2_\beta(0,\Tend;V)} \leq \alpha_0^{-1}\norm{g}_{L^2_\beta(0,\Tend;V^*)}.
\end{align}
\end{lemma}

\begin{proof}
We divide the proof into the existence argument and the weighted energy estimate.

First, we set
\begin{align*}
g_\beta(t):=e^{-\beta t}g(t).
\end{align*}
Then,
\begin{align*}
g_\beta\in L^2(0,\Tend;V^*),
\quad
\norm{g_\beta}_{L^2(0,\Tend;V^*)}
= \norm{g}_{L^2_\beta(0,\Tend;V^*)}.
\end{align*}
Let $J\in\mathcal L(V,V^*)$ be the pivot-space operator defined as
\begin{align*}
\pair{Jv}{w}:=(v,w)_H, \quad v,w\in V.
\end{align*}
By the continuous embedding $V\hookrightarrow H$, the bilinear form
\begin{align*}
  a_\beta(v,w):=a_0(v,w)+\beta(v,w)_H
\end{align*}
is bounded on $V\times V$. Furthermore, because $\beta\geq0$,
\begin{align*}
a_\beta(v,v)
=
a_0(v,v)+\beta\norm{v}_H^2
\geq
\alpha_0\norm{v}_V^2 \quad  \forall v\in V.
\end{align*}
Therefore, the standard theory of coercive parabolic problems on a Gelfand triple~\cite{Showalter1997,LionsMagenes1972} gives a unique $z\in\W(0,\Tend)$ satisfying
\begin{align}\label{eq:transformed-parabolic}
\partial_tz+A_0z+\beta Jz=g_\beta, \quad z(0)=y_0.
\end{align}
We return to the original unknown.  We define
\begin{align*}
y(t):=e^{\beta t}z(t).
\end{align*}
Multiplication by $e^{\beta t}$ is an automorphism of $\W(0,\Tend)$, so $y\in\W(0,\Tend)$.  Furthermore, in $L^2(0,\Tend;V^*)$,
\begin{align*}
\partial_ty = e^{\beta t}\bigl(\partial_tz+\beta Jz\bigr).
\end{align*}
Because $A_0y=e^{\beta t}A_0z$, equation \eqref{eq:transformed-parabolic} gives
\begin{align*}
\partial_ty+A_0y
&=
e^{\beta t}
\left(\partial_tz+\beta Jz+A_0z\right) = e^{\beta t}g_\beta = g.
\end{align*}
Also $y(0)=z(0)=y_0$.  Conversely, the transformation $z(t)=e^{-\beta t}y(t)$ maps any solution of the original problem to a solution of \eqref{eq:transformed-parabolic}.  The uniqueness of $z$ therefore implies the uniqueness of $y$.

It remains to prove the estimates.  Testing \eqref{eq:transformed-parabolic} by $z(t)$ gives, for almost every $t\in(0,\Tend)$,
\begin{align*}
\frac12\frac{\dd}{\dd t}\norm{z(t)}_H^2 +a_0 \left(z(t),z(t) \right) +\beta\norm{z(t)}_H^2
 = \pair{g_\beta(t)}{z(t)}.
\end{align*}
Using the coercivity of $a_0$ and Young's inequality,
\begin{align*}
\pair{g_\beta(t)}{z(t)}
&\leq
\norm{g_\beta(t)}_{V^*}\norm{z(t)}_V
\leq
\frac{1}{2\alpha_0}\norm{g_\beta(t)}_{V^*}^2
+
\frac{\alpha_0}{2}\norm{z(t)}_V^2,
\end{align*}
we have
\begin{align*}
\frac12\frac{\dd}{\dd t}\norm{z(t)}_H^2 +\beta\norm{z(t)}_H^2 +\frac{\alpha_0}{2}\norm{z(t)}_V^2
\leq \frac{1}{2\alpha_0}\norm{g_\beta(t)}_{V^*}^2.
\end{align*}
Integrating over $(0,t)$ and multiplying by two yield
\begin{align}\label{eq:z-energy}
\norm{z(t)}_H^2
&+ 2\beta\int_0^t\norm{z(r)}_H^2\,\dd r + \alpha_0\int_0^t\norm{z(r)}_V^2\,\dd r\\
&\leq
\norm{y_0}_H^2 + \alpha_0^{-1} \int_0^t\norm{g_\beta(r)}_{V^*}^2\,\dd r. \notag
\end{align}
For every $t\in[0,\Tend]$, discarding the two integral terms in \eqref{eq:z-energy} and enlarging the integral on the right-hand side to $(0,\Tend)$ give
\begin{align*}
\norm{z(t)}_H^2
\leq
\norm{y_0}_H^2 + \alpha_0^{-1} \norm{g_\beta}_{L^2(0,\Tend;V^*)}^2.
\end{align*}
Because
\begin{align*}
z(t)=e^{-\beta t}y(t),
\end{align*}
taking the supremum over $t\in[0,\Tend]$ proves \eqref{eq:weighted-parabolic-sup}. Next, setting $t=\Tend$ in \eqref{eq:z-energy} and discarding the other non-negative terms give
\begin{align*}
\alpha_0 \norm{z}_{L^2(0,\Tend;V)}^2
\leq
\norm{y_0}_H^2 + \alpha_0^{-1} \norm{g_\beta}_{L^2(0,\Tend;V^*)}^2.
\end{align*}
Because
\begin{align*}
\norm{z}_{L^2(0,\Tend;V)}
= \norm{y}_{L^2_\beta(0,\Tend;V)},
\end{align*}
this proves \eqref{eq:weighted-parabolic-L2}. Adding \eqref{eq:weighted-parabolic-sup} and \eqref{eq:weighted-parabolic-L2} gives \eqref{eq:weighted-parabolic-energy}.  Finally, if $y_0=0$, then \eqref{eq:weighted-parabolic-L2} becomes
\begin{align*}
\alpha_0 \norm{y}_{L^2_\beta(0,\Tend;V)}^2
\leq \alpha_0^{-1} \norm{g}_{L^2_\beta(0,\Tend;V^*)}^2.
\end{align*}
Taking square roots proves \eqref{eq:parabolic-Lipschitz}.
\end{proof}

\begin{remark}[Why the factor two does not affect the contraction]
The fixed-point argument uses only \eqref{eq:parabolic-Lipschitz}. The factor two in the combined energy estimate therefore affects only the constants in the subsequent full $\W(0,\Tend)$-bound.
\end{remark}

\subsection{Fixed-point construction}
We combine causality with coercivity.  Notice that the contraction parameter depends on $\kappa_\mu(\beta)$, whereas the inhomogeneous history term is allowed to depend on the full mass.  This distinction is essential: the past is prescribed data and does not enter the fixed-point difference.

\begin{theorem}[Weighted well-posedness]\label[theorem]{thm:wellposedness}
Let $\mu\in\M^0([0,\tau])$.  We choose $\beta\ge0$ such that
\begin{align}\label{eq:contraction-condition}
q_{\mu,\beta}:=\frac{\Lambda_1\kappa_\mu(\beta)}{\alpha_0}<1.
\end{align}
Then, \eqref{eq:problem} has a unique weak solution $u_\mu\in\W(0,\Tend)$.  Furthermore,
\begin{align}\label{eq:weighted-u-bound}
\norm{u_\mu}_{L^2_\beta(0,\Tend;V)}
\leq \frac{1}{1-q_{\mu,\beta}}
\Biggl [&\alpha_0^{-1/2}\norm{u_0}_H +\alpha_0^{-1}\norm{f}_{L^2_\beta(0,\Tend;V^*)}  \notag\\
&+\alpha_0^{-1}\Lambda_1\norm{\mu}_{\M([0,\tau])} \norm{\psi}_{L^2(-\tau,0;V)}\Biggr].
\end{align}
For each fixed $\mu\in\M^0([0,\tau])$, a weight $\beta$ satisfying \eqref{eq:contraction-condition} exists by \eqref{eq:kappa-decay}.
\end{theorem}

\begin{proof}
We set
\begin{align*}
X_\beta:=L^2_\beta(0,\Tend;V).
\end{align*}
From \Cref{lem:weighted-equivalence}, $X_\beta$ coincides as a vector space with $L^2(0,\Tend;V)$ and is equipped with an equivalent Hilbert norm. In particular, $X_\beta$ is complete. We define the fixed-point map.  From \eqref{eq:history-forcing-bound} and \Cref{lem:weighted-equivalence},
\begin{align*}
\norm{h_{\mu,\psi}}_{L^2_\beta(0,\Tend;V^*)}
&\leq
\norm{h_{\mu,\psi}}_{L^2(0,\Tend;V^*)} 
\leq
\Lambda_1 \norm{\mu}_{\M([0,\tau])} \norm{\psi}_{L^2(-\tau,0;V)}.
\end{align*}
Furthermore, for any $w\in X_\beta$, \Cref{lem:weighted-causal} gives
\begin{align*}
  \K_\mu^+w\in L^2_\beta(0,\Tend;V^*).
\end{align*}
Therefore,
\begin{align*}
f-h_{\mu,\psi}-\K_\mu^+w \in L^2_\beta(0,\Tend;V^*).
\end{align*}
By \Cref{lem:weighted-parabolic}, there exists a unique $\Phi w\in\W(0,\Tend)$ satisfying
\begin{subequations} \label{eq:fixed-point-map}
\begin{align}
\displaystyle
\partial_t(\Phi w)+A_0(\Phi w) &= f-h_{\mu,\psi}-\K_\mu^+w, \label{eq:fixed-point-map=a} \\
 (\Phi w)(0)&=u_0. \label{eq:fixed-point-map=b}
\end{align}
\end{subequations}
Because
\begin{align*}
\W(0,\Tend)\subset L^2(0,\Tend;V)=X_\beta,
\end{align*}
this defines a map
\begin{align*}
  \Phi:X_\beta\longrightarrow X_\beta.
\end{align*}
We prove that $\Phi$ is a contraction.  Let $w_1,w_2\in X_\beta$ and set
\begin{align*}
z:=\Phi w_1-\Phi w_2.
\end{align*}
Subtracting the two equations in \eqref{eq:fixed-point-map}, we have
\begin{align*}
\partial_tz+A_0z
&= -\K_\mu^+(w_1-w_2), \\
z(0)&=0.
\end{align*}
Therefore, \eqref{eq:parabolic-Lipschitz} and \eqref{eq:weighted-causal} yield
\begin{align*}
\norm{\Phi w_1-\Phi w_2}_{L^2_\beta(0,\Tend;V)}
&\leq
\alpha_0^{-1} \norm{\K_\mu^+(w_1-w_2)}_{L^2_\beta(0,\Tend;V^*)}\\
&\leq
\frac{\Lambda_1\kappa_\mu(\beta)}{\alpha_0} \norm{w_1-w_2}_{L^2_\beta(0,\Tend;V)}\\
&= q_{\mu,\beta} \norm{w_1-w_2}_{L^2_\beta(0,\Tend;V)}.
\end{align*}
Because $q_{\mu,\beta}<1$, the map $\Phi$ is a strict contraction on the complete space $X_\beta$.  The Banach fixed-point theorem therefore gives a unique $u_\mu\in X_\beta$ such that
\begin{align*}
\Phi u_\mu=u_\mu.
\end{align*}
Because $\Phi u_\mu\in\W(0,\Tend)$, the fixed point satisfies
\begin{align*}
u_\mu\in\W(0,\Tend).
\end{align*}
Its fixed-point equation is
\begin{align*}
\partial_tu_\mu+A_0u_\mu
=
f-h_{\mu,\psi}-\K_\mu^+u_\mu,
\quad
u_\mu(0)=u_0.
\end{align*}
By the history decomposition \eqref{eq:history-splitting}, this is precisely the weak formulation of \eqref{eq:problem}.  Thus, $u_\mu$ is a weak solution.

To prove uniqueness among all weak solutions, let $v\in\W(0,\Tend)$ be any weak solution of \eqref{eq:problem}. Then, $v\in X_\beta$, and \eqref{eq:history-splitting} shows that $v$ satisfies \eqref{eq:fixed-point-map} with $w=v$.  Therefore, $v=\Phi v$.  Since $\Phi$ has only one fixed point in $X_\beta$, we conclude that $v=u_\mu$.

It remains to establish the estimate.  We define
\begin{align*}
F_\mu
:= f-h_{\mu,\psi}-\K_\mu^+u_\mu.
\end{align*}
Applying \eqref{eq:weighted-parabolic-L2} to
\begin{align*}
\partial_tu_\mu+A_0u_\mu=F_\mu,
\quad
u_\mu(0)=u_0,
\end{align*}
gives
\begin{align*}
\alpha_0
\norm{u_\mu}_{L^2_\beta(0,\Tend;V)}^2
\leq
\norm{u_0}_H^2
+
\alpha_0^{-1}
\norm{F_\mu}_{L^2_\beta(0,\Tend;V^*)}^2.
\end{align*}
Taking square roots and using $\sqrt{a+b}\leq\sqrt a+\sqrt b$ $(a,b \geq 0, a,b \in \R)$, we have
\begin{align*}
\norm{u_\mu}_{L^2_\beta(0,\Tend;V)}
\leq
\alpha_0^{-1/2}\norm{u_0}_H
+
\alpha_0^{-1}
\norm{F_\mu}_{L^2_\beta(0,\Tend;V^*)}.
\end{align*}
The triangle inequality, the estimate for the prescribed-history forcing, and \eqref{eq:weighted-causal} now give
\begin{align*}
\norm{u_\mu}_{L^2_\beta(0,\Tend;V)}
&\leq \alpha_0^{-1/2}\norm{u_0}_H + \alpha_0^{-1} \norm{f}_{L^2_\beta(0,\Tend;V^*)}\\
&\quad \quad +
  \alpha_0^{-1} \norm{h_{\mu,\psi}}_{L^2_\beta(0,\Tend;V^*)} +
  \alpha_0^{-1}
  \norm{\K_\mu^+u_\mu}_{L^2_\beta(0,\Tend;V^*)}\\
&\quad \leq
  \alpha_0^{-1/2}\norm{u_0}_H
  +
  \alpha_0^{-1}
  \norm{f}_{L^2_\beta(0,\Tend;V^*)}\\ 
&\quad \quad + \alpha_0^{-1}\Lambda_1 \norm{\mu}_{\M([0,\tau])} \norm{\psi}_{L^2(-\tau,0;V)} + q_{\mu,\beta} \norm{u_\mu}_{L^2_\beta(0,\Tend;V)}.
\end{align*}
Because $q_{\mu,\beta}<1$, the last term can be absorbed into the left-hand side. Therefore,
\begin{align*}
  \norm{u_\mu}_{L^2_\beta(0,\Tend;V)}
  \leq
  \frac{1}{1-q_{\mu,\beta}}
  \Bigl(
  &\alpha_0^{-1/2}\norm{u_0}_H
  +
  \alpha_0^{-1}
  \norm{f}_{L^2_\beta(0,\Tend;V^*)}\\
  &+
  \alpha_0^{-1}\Lambda_1
  \norm{\mu}_{\M([0,\tau])}
  \norm{\psi}_{L^2(-\tau,0;V)}
  \Bigr),
\end{align*}
which is \eqref{eq:weighted-u-bound}.

Finally, because $\mu\in\M^0([0,\tau])$, \eqref{eq:kappa-decay} implies
\begin{align*}
  \kappa_\mu(\beta)\longrightarrow0
  \quad\text{as }\beta\to\infty.
\end{align*}
Therefore, a value of $\beta\geq0$ satisfying $q_{\mu,\beta}<1$ always exists.
\end{proof}


\begin{corollary}[Full energy estimate]
\label[corollary]{cor:full-energy}
Under the assumptions of \Cref{thm:wellposedness}, there exists a constant $ C = C\left( \alpha_0,\Lambda_0,\Lambda_1,\beta,\Tend, \norm{\mu}_{\M([0,\tau])},q_{\mu,\beta} \right)>0$ such that
\begin{align}
&\norm{u_\mu}_{\mathcal{C}([0,\Tend];H)}
  +\norm{u_\mu}_{L^2(0,\Tend;V)}
  +\norm{\partial_tu_\mu}_{L^2(0,\Tend;V^*)} \notag\\
&\quad\leq
  C\left(
    \norm{u_0}_H
    +\norm{f}_{L^2(0,\Tend;V^*)}
    +\norm{\psi}_{L^2(-\tau,0;V)}
  \right). \label{eq:full-energy-estimate}
\end{align}
\end{corollary}

\begin{proof}
We set
\begin{align*}
  D_{\mu,\beta}
  &:=
  \alpha_0^{-1/2}\norm{u_0}_H
  +\alpha_0^{-1}
  \norm{f}_{L^2_\beta(0,\Tend;V^*)}\\
  &\quad+
  \alpha_0^{-1}\Lambda_1
  \norm{\mu}_{\M([0,\tau])}
  \norm{\psi}_{L^2(-\tau,0;V)}.
\end{align*}
From \eqref{eq:weighted-u-bound},
\begin{align}\label{eq:full-energy-u-intermediate}
  \norm{u_\mu}_{L^2_\beta(0,\Tend;V)}
  \leq
  \frac{D_{\mu,\beta}}{1-q_{\mu,\beta}}.
\end{align}
We define
\begin{align*}
  G_\mu
  :=
  f-h_{\mu,\psi}-\K_\mu^+u_\mu.
\end{align*}
Then,
\begin{align*}
  \partial_tu_\mu+A_0u_\mu=G_\mu,
  \quad
  u_\mu(0)=u_0.
\end{align*}
By \eqref{eq:history-forcing-bound}, \eqref{eq:weighted-causal}, and \eqref{eq:full-energy-u-intermediate},
\begin{align*}
&\norm{G_\mu}_{L^2_\beta(0,\Tend;V^*)} \\
&\quad \leq
\norm{f}_{L^2_\beta(0,\Tend;V^*)}
+
\Lambda_1\norm{\mu}_{\M([0,\tau])} \norm{\psi}_{L^2(-\tau,0;V)}
+
\Lambda_1\kappa_\mu(\beta) \frac{D_{\mu,\beta}}{1-q_{\mu,\beta}}.
\end{align*}
Applying \eqref{eq:weighted-parabolic-sup} to $\partial_tu_\mu+A_0u_\mu=G_\mu$ yields
\begin{align*}
\sup_{0\leq t\leq\Tend}
e^{-\beta t}\norm{u_\mu(t)}_H
\leq
\left(
 \norm{u_0}_H^2 + \alpha_0^{-1} \norm{G_\mu}_{L^2_\beta(0,\Tend;V^*)}^2
\right)^{1/2}.
\end{align*}
Furthermore, the equation and the boundedness of $A_0$ yield
\begin{align*}
  \norm{\partial_tu_\mu}_{L^2_\beta(0,\Tend;V^*)}
  &\leq
  \norm{A_0u_\mu}_{L^2_\beta(0,\Tend;V^*)}
  +
  \norm{G_\mu}_{L^2_\beta(0,\Tend;V^*)}\\
  &\leq
  \Lambda_0
  \norm{u_\mu}_{L^2_\beta(0,\Tend;V)}
  +
  \norm{G_\mu}_{L^2_\beta(0,\Tend;V^*)}.
\end{align*}
\Cref{lem:weighted-equivalence} and
\begin{align*}
  \sup_{0\leq t\leq\Tend}\norm{u_\mu(t)}_H
  \leq
  e^{\beta\Tend}
  \sup_{0\leq t\leq\Tend}
  e^{-\beta t}\norm{u_\mu(t)}_H
\end{align*}
convert the weighted estimates into unweighted estimates. Because
\begin{align*}
  \norm{f}_{L^2_\beta(0,\Tend;V^*)}
  \leq
  \norm{f}_{L^2(0,\Tend;V^*)},
\end{align*}
the resulting constant depends only on the parameters displayed in the statement.  This proves \eqref{eq:full-energy-estimate}.
\end{proof}

\begin{remark}[Interpretation of the contraction condition]
The condition $q_{\mu,\beta}<1$ is not a small-data condition and does not restrict the size of a fixed genuinely retarded measure $\mu\in\M^0([0,\tau])$.  For any such measure, it can be enforced by increasing the exponential weight $\beta$. Failure of the inequality for a particular value of $\beta$ means only that the fixed-point map is not contractive in that particular weighted norm; it does not imply non-uniqueness or ill-posedness of the linear equation.  Measures with an atom at the origin are structurally different, because that atom acts as an instantaneous term and should instead be absorbed into the principal form as described in \Cref{rem:atom-zero}.
\end{remark}

\begin{theorem}[Dependence on the data]\label[theorem]{thm:data-stability}
Fix $\mu\in\M^0([0,\tau])$ and choose $\beta$ with $q_{\mu,\beta}<1$.  Let $u$ and $\bar u$ be the solutions associated with data $(u_0,f,\psi)$ and $(\bar u_0,\bar f,\bar\psi)$, respectively.  Then,
\begin{align}\label{eq:data-stability}
\norm{u-\bar u}_{L^2_\beta(0,\Tend;V)}
\leq \frac{1}{1-q_{\mu,\beta}}
\Biggl [&\alpha_0^{-1/2}\norm{u_0-\bar u_0}_H
+\alpha_0^{-1}\norm{f-\bar f}_{L^2_\beta(0,\Tend;V^*)}\notag\\
&+\alpha_0^{-1}\Lambda_1\norm{\mu}_{\M([0,\tau])}
\norm{\psi-\bar\psi}_{L^2(-\tau,0;V)}\Biggr].
\end{align}
Furthermore, there exists a constant $C_{\mu,\beta,\Tend}>0$, depending only on $\alpha_0$, the continuity bounds of $A_0$ and $A_1$, $\norm{\mu}_{\M([0,\tau])}$, $\beta$, $\Tend$, and $q_{\mu,\beta}$, such that
\begin{align*}
\norm{u-\bar u}_{\W(0,\Tend)}
+
\norm{u-\bar u}_{C([0,\Tend];H)}
\leq
C_{\mu,\beta,\Tend}\,\mathcal D_\beta,
\end{align*}
where
\begin{align*}
\mathcal D_\beta
&:=
\norm{u_0-\bar u_0}_H
+
\norm{f-\bar f}_{L^2_\beta(0,\Tend;V^*)}
+
\norm{\mu}_{\M([0,\tau])} \norm{\psi-\bar\psi}_{L^2(-\tau,0;V)}.
\end{align*}
\end{theorem}

\begin{proof}
We set $e:=u-\bar u$. By the linearity of the prescribed-history forcing and of the causal history operator,
\begin{align*}
h_{\mu,\psi}-h_{\mu,\bar\psi}
&= h_{\mu,\psi-\bar\psi}, \quad 
\K_\mu^+u-\K_\mu^+\bar u
= \K_\mu^+e.
\end{align*}
Subtracting the two weak equations and using \eqref{eq:history-splitting}, we obtain
\begin{align*}
\partial_te+A_0e
&=
(f-\bar f)-h_{\mu,\psi-\bar\psi}-\K_\mu^+e,\\
e(0)&=u_0-\bar u_0.
\end{align*}
Applying the $L^2_\beta(0,\Tend;V)$-estimate of \Cref{lem:weighted-parabolic} gives
\begin{align*}
\norm{e}_{L^2_\beta(0,\Tend;V)}
\leq{}&
\alpha_0^{-1/2}\norm{u_0-\bar u_0}_H\\
&+
\alpha_0^{-1}
\norm{ (f-\bar f)-h_{\mu,\psi-\bar\psi}-\K_\mu^+e }_{L^2_\beta(0,\Tend;V^*)}.
\end{align*}
Using the triangle inequality, \eqref{eq:history-forcing-bound}, and \Cref{lem:weighted-causal}, we have
\begin{align*}
\norm{e}_{L^2_\beta(0,\Tend;V)}
&\leq
\alpha_0^{-1/2}\norm{u_0-\bar u_0}_H
+
\alpha_0^{-1}
\norm{f-\bar f}_{L^2_\beta(0,\Tend;V^*)}\\
&\ +
\alpha_0^{-1}\Lambda_1
\norm{\mu}_{\M([0,\tau])}
\norm{\psi-\bar\psi}_{L^2(-\tau,0;V)}
+
\frac{\Lambda_1\kappa_\mu(\beta)}{\alpha_0} \norm{e}_{L^2_\beta(0,\Tend;V)}.
\end{align*}
Because the last coefficient is $q_{\mu,\beta}<1$, absorbing the last term into the left-hand side proves \eqref{eq:data-stability}.

The difference $e=u-\bar u$ satisfies the same type of equation as in \Cref{cor:full-energy}, with initial datum $u_0-\bar u_0$, forcing $f-\bar f$, and prescribed history $\psi-\bar\psi$.  Repeating the proof of \Cref{cor:full-energy} for this difference equation, while retaining the weighted norm $\norm{f-\bar f}_{L^2_\beta(0,\Tend;V^*)}$, gives the asserted estimate in $\W(0,\Tend)\cap \mathcal{C}([0,\Tend];H)$.
\end{proof}

\begin{corollary}[Uniform bounds for a kernel family]
\label[corollary]{cor:uniform-family}
Let $\mathfrak K\subset\M^0([0,\tau])$.  Suppose that there exist $\beta\geq0$ and $q_*<1$ such that
\begin{align*}
  \sup_{\mu\in\mathfrak K}
  \frac{\Lambda_1\kappa_\mu(\beta)}{\alpha_0}
  \leq q_*,
  \quad
  M_*:=
  \sup_{\mu\in\mathfrak K}
  \norm{\mu}_{\M([0,\tau])}
  <\infty.
\end{align*}
Then, the estimates in \Cref{thm:wellposedness,cor:full-energy,thm:data-stability} hold uniformly for $\mu\in\mathfrak K$.  In particular,
\begin{align*}
&\norm{u_\mu}_{L^2_\beta(0,\Tend;V)} \\
&\quad \leq \frac{1}{1-q_*}
\left[
\alpha_0^{-1/2}\norm{u_0}_H
  +\alpha_0^{-1}
  \norm{f}_{L^2_\beta(0,\Tend;V^*)}+\alpha_0^{-1}\Lambda_1M_*
  \norm{\psi}_{L^2(-\tau,0;V)}
  \right].
\end{align*}
The corresponding full-energy and data-stability estimates in $\W(0,\Tend)\cap \mathcal{C}([0,\Tend];H)$ are also uniform over $\mu\in\mathfrak K$.
\end{corollary}

\begin{proof}
For any $\mu\in\mathfrak K$, the assumptions give
\begin{align*}
q_{\mu,\beta}\leq q_*<1,
\quad
\norm{\mu}_{\M([0,\tau])}\leq M_*.
\end{align*}
Consequently,
\begin{align*}
  \frac{1}{1-q_{\mu,\beta}}
  \leq
  \frac{1}{1-q_*},
  \quad
  \Lambda_1\kappa_\mu(\beta)
  =
  \alpha_0q_{\mu,\beta}
  \leq
  \alpha_0q_*.
\end{align*}
Substituting the first of these bounds and $\norm{\mu}_{\M([0,\tau])}\leq M_*$ into \eqref{eq:weighted-u-bound} gives
\begin{align*}
  \norm{u_\mu}_{L^2_\beta(0,\Tend;V)}
  \leq
  \frac{1}{1-q_*}
  \Bigl[
  &\alpha_0^{-1/2}\norm{u_0}_H
  +\alpha_0^{-1}
  \norm{f}_{L^2_\beta(0,\Tend;V^*)}\\
  &+
  \alpha_0^{-1}\Lambda_1M_*
  \norm{\psi}_{L^2(-\tau,0;V)}
  \Bigr].
\end{align*}
We next consider the full-energy estimate.  In the proof of \Cref{cor:full-energy}, the kernel-dependent quantities occur only through
\begin{align*}
  \frac{1}{1-q_{\mu,\beta}},
  \quad
  \Lambda_1\kappa_\mu(\beta),
  \quad
  \norm{\mu}_{\M([0,\tau])}.
\end{align*}
The preceding uniform bounds therefore show that the weighted $L^2(0,\Tend;V)$-, $L^2(0,\Tend;V^*)$-, and continuous $H$-estimates obtained there have constants independent of $\mu\in\mathfrak K$. Because the weight $\beta$ is common to the whole family, the weighted-to-unweighted norm equivalences introduce only the common factor $e^{\beta\Tend}$. Therefore, the full-energy estimate is uniform over $\mathfrak K$, with a constant depending only on
\begin{align*}
  \alpha_0,\Lambda_0,\Lambda_1,\beta,\Tend,q_*,
  \ \text{and}\ M_*.
\end{align*}
Finally, substituting $q_{\mu,\beta}\leq q_*$ and $\norm{\mu}_{\M([0,\tau])}\leq M_*$ into \eqref{eq:data-stability} gives the uniform $L^2_\beta(0,\Tend;V)$ data-stability estimate.  Applying the same full-energy argument to the difference equation for $u-\bar u$ gives the corresponding uniform estimate in $\W(0,\Tend)\cap\mathcal{C}([0,\Tend];H)$.
\end{proof}

\section{Delayed diffusion fluxes: a concrete PDE realisation}\label{sec:pde-realisation}
The abstract formulation is designed for situations in which the delayed quantity contains spatial derivatives.  This section spells out that case in detail.  Besides providing a concrete model, it explains why an $H$-valued delay theory cannot simply be invoked: even the variational Dirichlet Laplacian does not admit a bounded extension from the pivot space into the energy dual.

\subsection{The model and its weak formulation}
Let $\Omega\subset\R^d$, $d\in\{1,2,3\}$, be a bounded Lipschitz domain. We set
\begin{align*}
  H=L^2(\Omega),\quad V=H_0^1(\Omega),\quad V^*=H^{-1}(\Omega),
\end{align*}
and equip $V$ with
\begin{align*}
  \norm{v}_V:=\norm{\nabla v}_{L^2(\Omega;\R^d)}.
\end{align*}
By the Poincar\'e inequality, this is a norm equivalent to the usual $H^1(\Omega)$-norm on $H_0^1(\Omega)$.

Let $B_0,B_1\in L^\infty(\Omega;\R^{d\times d})$.  Throughout this section, the coefficient norm is the essential supremum of the Euclidean operator norm,
\begin{align}\label{eq:coefficient-operator-norm}
  \norm{B_i}_{L^\infty(\Omega;\mathcal L(\R^d))}
  :=\mathop{\operatorname{ess\,sup}}_{x\in\Omega}
    \sup_{0\ne\xi\in\R^d}
    \frac{|B_i(x)\xi|}{|\xi|},
  \quad i\in\{0,1\}.
\end{align}
We assume that there is a number $\lambda_0>0$ such that, for almost every $x\in\Omega$ and any $\xi\in\R^d$,
\begin{align}\label{eq:B0-ellipticity}
  \xi^\top B_0(x)\xi\ge \lambda_0|\xi|^2.
\end{align}
No symmetry is required of $B_0$, and no symmetry, ellipticity, or sign condition is imposed on $B_1$.  The corresponding bilinear forms are
\begin{align}\label{eq:diffusion-forms}
  a_i(w,v)=\int_\Omega B_i(x)\nabla w\cdot\nabla v\,\dd x,
  \quad i\in\{0,1\}.
\end{align}
For a finite retarded measure $\mu\in\M^0([0,\tau])$, the strong notation for the problem is
\begin{align}
  \partial_tu
  -\nabla\!\cdot(B_0\nabla u)
  -\int_{[0,\tau]}
       \nabla\!\cdot\left(B_1\nabla\tilde u(t-s)\right)
       \,\dd\mu(s)
  &=f
  &&\text{in }\Omega\times(0,\Tend), \label{eq:delayed-diffusion-strong}\\
  u&=0
  &&\text{on }\partial\Omega\times(0,\Tend),\label{eq:delayed-diffusion-bc}\\
  u(0)&=u_0
  &&\text{in }L^2(\Omega),\label{eq:delayed-diffusion-ic}
\end{align}
with prescribed past $\psi$ on $(-\tau,0)$.  Formula \eqref{eq:delayed-diffusion-strong} is only mnemonic: neither the divergence of the delayed flux nor the measure integral is assumed to be an $L^2(\Omega)$-valued function. The rigorous formulation is the following.  We seek $u\in\W(0,\Tend)$, with the joined trajectory $\tilde u$ from \Cref{lem:joined-trajectory}, such that
\begin{align}\label{eq:delayed-diffusion-Vstar}
  \partial_tu(t)+A_0u(t)+\K_\mu\tilde u(t)=f(t)
  \quad\text{in }H^{-1}(\Omega)
\end{align}
for almost every $t\in(0,\Tend)$ and $u(0)=u_0$ in $L^2(\Omega)$. Equivalently,
\begin{align}\label{eq:delayed-diffusion-weak}
  \pair{\partial_tu(t)}{v}
  &+\int_\Omega B_0\nabla u(t)\cdot\nabla v\,\dd x\notag\\
  &+\int_{[0,\tau]}
       \int_\Omega B_1\nabla\tilde u(t-s)\cdot\nabla v\,\dd x
       \,\dd\mu(s)
   =\pair{f(t)}{v}
\end{align}
for any $v\in H_0^1(\Omega)$ and for almost every $t\in(0,\Tend)$.  From \Cref{lem:measure-convolution}, this formulation requires only
\begin{align*}
  \tilde u\in L^2(-\tau,\Tend;H_0^1(\Omega)).
\end{align*}
In particular, an atomic kernel $m\delta_{\tau_d}$ produces the translated term
\begin{align*}
  m\int_\Omega B_1\nabla\tilde u(t-\tau_d)\cdot\nabla v\,\dd x.
\end{align*}
The map $t\mapsto\tilde u(t-\tau_d)$ is a well-defined $L^2$-equivalence class, and changing the representative of the past on a Lebesgue-null set changes this translated term only for a Lebesgue-null set of present times.

\begin{proposition}[Verification of the abstract assumptions]\label[proposition]{prop:pde-assumptions}
Under \eqref{eq:B0-ellipticity}, the forms in \eqref{eq:diffusion-forms} satisfy, for all $w,v\in V$,
\begin{align*}
  |a_i(w,v)|
  &\leq
  \norm{B_i}_{L^\infty(\Omega;\mathcal L(\R^d))}
  \norm{w}_V\norm{v}_V,
  \quad i\in\{0,1\},\\
  a_0(v,v) &\geq \lambda_0\norm{v}_V^2.
\end{align*}
Then, the delayed-diffusion problem \eqref{eq:delayed-diffusion-Vstar} is an instance of \eqref{eq:problem}, where one may take
\begin{align*}
  \alpha_0=\lambda_0,
  \quad
  \Lambda_i=
  \norm{B_i}_{L^\infty(\Omega;\mathcal L(\R^d))},
  \quad i\in\{0,1\}.
\end{align*}
\end{proposition}

\begin{proof}
For almost every $x\in\Omega$,
\begin{align*}
  |B_i(x)\nabla w(x)\cdot\nabla v(x)|
  \leq
  \norm{B_i(x)}_{\mathcal L(\R^d)}
  |\nabla w(x)|\,|\nabla v(x)|.
\end{align*}
Integration and the Cauchy--Schwarz inequality give the boundedness of $a_i$.  Taking $\xi=\nabla v(x)$ in \eqref{eq:B0-ellipticity} and integrating gives
\begin{align*}
  a_0(v,v)
  =\int_\Omega (B_0\nabla v)\cdot\nabla v\,\dd x
  \geq \lambda_0\norm{\nabla v}_{L^2(\Omega;\R^d)}^2.
\end{align*}
The operators generated by the two forms, therefore, satisfy the abstract hypotheses in \eqref{eq:form-bounded}--\eqref{eq:a0-coercive}, and \eqref{eq:delayed-diffusion-Vstar} is exactly \eqref{eq:problem} in the present choice of $V$, $H$, and $V^*$.
\end{proof}

\subsection{Why the delayed elliptic operator is not pivot-space bounded}
The next observation is central to the scope of the paper.  A reaction delay may act boundedly on $L^2(\Omega)$, whereas a delayed diffusion flux need not.

\begin{proposition}[Failure of an $H\to V^*$ extension]\label[proposition]{prop:not-H-bounded}
Take $B_1=I$.  The variational Dirichlet Laplacian
\begin{align*}
  A_1=-\Delta:H_0^1(\Omega)\longrightarrow H^{-1}(\Omega)
\end{align*}
does not admit an extension
\begin{align*}
  \widetilde A_1\in
  \mathcal L\bigl(L^2(\Omega),H^{-1}(\Omega)\bigr)
\end{align*}
whose restriction to $H_0^1(\Omega)$ agrees with $A_1$.  Consequently, if $\iota:L^2(\Omega)\hookrightarrow H^{-1}(\Omega)$ denotes the canonical embedding, there is also no operator
\begin{align*}
  \widehat A_1\in
  \mathcal L\bigl(L^2(\Omega),L^2(\Omega)\bigr)
\end{align*}
such that
\begin{align*}
  \iota\widehat A_1v=A_1v
  \quad\text{in }H^{-1}(\Omega)
  \quad\text{for any }v\in H_0^1(\Omega).
\end{align*}
\end{proposition}

\begin{proof}
Because $\Omega$ is bounded and Lipschitz, the embedding $H_0^1(\Omega)\hookrightarrow L^2(\Omega)$ is compact.  By the compact embedding $H_0^1(\Omega)\hookrightarrow L^2(\Omega)$ and the spectral theorem for the Dirichlet Laplacian, there exist weak eigenpairs $(\phi_n,\lambda_n)$ satisfying
\begin{align*}
  \phi_n\in H_0^1(\Omega),\qquad
  \norm{\phi_n}_{L^2(\Omega)}=1,\qquad
  \lambda_n\longrightarrow\infty,
\end{align*}
and
\begin{align}\label{eq:weak-dirichlet-eigenproblem}
  \int_\Omega\nabla\phi_n\cdot\nabla v\,\dd x
  =\lambda_n\int_\Omega\phi_nv\,\dd x
  \quad\forall v\in H_0^1(\Omega).
\end{align}
Equip $H^{-1}(\Omega)$ with the dual norm induced by $\norm{v}_V=\norm{\nabla v}_{L^2(\Omega;\R^d)}$.  Because the variational Laplacian is the Riesz map associated with the gradient inner product,
\begin{align*}
  \norm{A_1\phi_n}_{H^{-1}(\Omega)}
  &=\sup_{0\ne v\in H_0^1(\Omega)}
    \frac{\left|\int_\Omega\nabla\phi_n\cdot\nabla v\,\dd x\right|}
         {\norm{\nabla v}_{L^2(\Omega;\R^d)}}
  =\norm{\nabla\phi_n}_{L^2(\Omega;\R^d)}.
\end{align*}
Testing \eqref{eq:weak-dirichlet-eigenproblem} with $v=\phi_n$ gives
\begin{align*}
  \norm{\nabla\phi_n}_{L^2(\Omega;\R^d)}^2=\lambda_n,
\end{align*}
so
\begin{align*}
  \norm{A_1\phi_n}_{H^{-1}(\Omega)}=\lambda_n^{1/2}\longrightarrow\infty,
  \quad
  \norm{\phi_n}_{L^2(\Omega)}=1.
\end{align*}
This excludes a bounded extension from $L^2(\Omega)$ to $H^{-1}(\Omega)$. Indeed, any such extension would satisfy
\begin{align*}
\norm{A_1\phi_n}_{H^{-1}(\Omega)}
\leq
\norm{\widetilde A_1}_{\mathcal L(L^2(\Omega),H^{-1}(\Omega))}
\norm{\phi_n}_{L^2(\Omega)},
\end{align*}
which contradicts the preceding divergence. Finally, if an operator $\widehat A_1$ with the stated compatibility property existed, then the composition $\iota\widehat A_1$ would belong to $\mathcal L(L^2(\Omega),H^{-1}(\Omega))$ and would be the extension just excluded.
\end{proof}

\begin{remark}[What is gained by the energy-space formulation]
For $B_1=I$, the history term is a measure superposition of delayed variational Laplacians.  The joined trajectory needs only one spatial derivative in the $L^2$ sense, and the equation is solved in $H^{-1}(\Omega)$.  Replacing the history space by $\mathcal C([-\tau,0];L^2(\Omega))$ would provide time continuity but would not by itself give the spatial $H_0^1(\Omega)$-regularity needed to apply the variational Laplacian
\begin{align*}
  -\Delta:H_0^1(\Omega)\longrightarrow H^{-1}(\Omega).
\end{align*}
Additional spatial regularity would therefore be required.  Such regularity is not supplied by the basic energy estimate and is not part of the weak parabolic solution concept used here.
\end{remark}

\subsection{Well-posedness for delayed diffusion}
We present the direct PDE consequences of the abstract results.  They are stated separately so that the hypotheses can be checked without translating between the strong and variational notations.

\begin{theorem}[Well-posedness of the delayed-diffusion problem]\label[theorem]{thm:pde-wellposedness}
Assume \eqref{eq:B0-ellipticity}, let $\mu\in\M^0([0,\tau])$, and take
\begin{align*}
  u_0&\in L^2(\Omega), \quad
  f \in L^2(0,\Tend;H^{-1}(\Omega)), \quad
  \psi \in L^2(-\tau,0;H_0^1(\Omega)).
\end{align*}
Then, \eqref{eq:delayed-diffusion-Vstar}, equivalently \eqref{eq:delayed-diffusion-weak}, has a unique solution
\begin{align*}
  u&\in L^2(0,\Tend;H_0^1(\Omega)),\quad
  \partial_tu \in L^2(0,\Tend;H^{-1}(\Omega)),\quad
  u \in \mathcal{C}([0,\Tend];L^2(\Omega)).
\end{align*}
If $\beta\ge0$ is chosen so that
\begin{align}\label{eq:pde-weight-condition}
  \norm{B_1}_{L^\infty(\Omega;\mathcal L(\R^d))}
  \kappa_\mu(\beta)<\lambda_0,
\end{align}
then
\begin{align}\label{eq:pde-weighted-bound}
  \norm{u}_{L^2_\beta(0,\Tend;H_0^1(\Omega))}
  \le \frac{1}{1-q_{\mu,\beta}}
  \Biggl [&\lambda_0^{-1/2}\norm{u_0}_{L^2(\Omega)}
   +\lambda_0^{-1}\norm{f}_{L^2_\beta(0,\Tend;H^{-1}(\Omega))}\notag\\
   &+\lambda_0^{-1}
     \norm{B_1}_{L^\infty(\Omega;\mathcal L(\R^d))}
     \norm{\mu}_{\M([0,\tau])}\notag\\
   &\quad\times
     \norm{\psi}_{L^2(-\tau,0;H_0^1(\Omega))}\Biggr],
\end{align}
where
\begin{align*}
  q_{\mu,\beta}
  :=\frac{
  \norm{B_1}_{L^\infty(\Omega;\mathcal L(\R^d))}
  \kappa_\mu(\beta)}{\lambda_0}.
\end{align*}
For each fixed $\mu\in\M^0([0,\tau])$, a value of $\beta$ satisfying \eqref{eq:pde-weight-condition} exists without any smallness assumption on $\norm{\mu}_{\M([0,\tau])}$.  The required value of $\beta$ may depend on $\mu$, including its total variation and the distribution of $|\mu|$ near the origin.
\end{theorem}

\begin{proof}
From \Cref{prop:pde-assumptions}, the abstract hypotheses \eqref{eq:form-bounded}--\eqref{eq:a0-coercive} hold with
\begin{align*}
  V=H_0^1(\Omega),\quad
  H=L^2(\Omega),\quad
  V^*=H^{-1}(\Omega),
\end{align*}
and one may take
\begin{align*}
  \alpha_0=\lambda_0,\qquad
  \Lambda_i=
  \norm{B_i}_{L^\infty(\Omega;\mathcal L(\R^d))},
  \quad i\in\{0,1\}.
\end{align*}
Condition \eqref{eq:pde-weight-condition} is therefore equivalent to $q_{\mu,\beta}<1$. Therefore, \Cref{thm:wellposedness} gives a unique solution
\begin{align*}
  u\in\W(0,\Tend)
\end{align*}
and yields \eqref{eq:pde-weighted-bound}.  The embedding
\begin{align*}
  \W(0,\Tend)\hookrightarrow
  \mathcal C([0,\Tend];L^2(\Omega))
\end{align*}
gives the stated time continuity; the corresponding full-energy estimate also follows from \Cref{cor:full-energy}. Finally, because $\mu\in\M^0([0,\tau])$, one has $|\mu|(\{0\})=0$.  Thus, \Cref{lem:kappa-properties} gives $\kappa_\mu(\beta)\to0$ as $\beta\to\infty$, so a value of $\beta$ satisfying \eqref{eq:pde-weight-condition} exists for any fixed kernel.
\end{proof}


\section{Stability with respect to the kernel}\label{sec:stability}
This section studies the dependence of the solution on the delay kernel.  By rewriting the difference of two solutions as a residual equation, we obtain stability in total variation and, for non-negative kernels, strong convergence under narrow convergence.

\subsection{The residual principle}
The kernel perturbation is evaluated along a single reference trajectory. This asymmetric formulation is useful: the kernel on the left determines the weighted resolvent margin, whereas the kernel on the right provides the reference trajectory whose translations are compared.

Let $\mu,\nu\in\M^0([0,\tau])$, and let $u_\mu,u_\nu$ denote the solutions \eqref{eq:problem} with the same data $(u_0,f,\psi)$.  We abbreviate
\begin{align}\label{eq:resolvent-gap}
  d_{\mu,\beta}:=\alpha_0-\Lambda_1\kappa_\mu(\beta).
\end{align}
The quantity $d_{\mu,\beta}$ is the positive margin left after the weighted causal perturbation has been absorbed.

\begin{theorem}[Residual kernel-stability estimate]\label[theorem]{thm:residual-stability}
Let $\beta \geq 0$ satisfy $d_{\mu,\beta}>0$. We set
\begin{align*}
R_{\mu,\nu}:=\K_{\mu-\nu}\tilde u_\nu.
\end{align*}
Then, $R_{\mu,\nu}\in L^2_\beta(0,\Tend;V^*)$ and
\begin{align}
  \norm{u_\mu-u_\nu}_{L^2_\beta(0,\Tend;V)}
  &\leq d_{\mu,\beta}^{-1}
  \norm{R_{\mu,\nu}}_{L^2_\beta(0,\Tend;V^*)},
  \label{eq:residual-stability}\\
  \sup_{0\le t\le \Tend}e^{-\beta t}
  \norm{u_\mu(t)-u_\nu(t)}_H
  &\leq \frac{\sqrt{\alpha_0}}{d_{\mu,\beta}}
  \norm{R_{\mu,\nu}}_{L^2_\beta(0,\Tend;V^*)},
  \label{eq:residual-H}\\
  \norm{\partial_tu_\mu-\partial_tu_\nu}_{L^2_\beta(0,\Tend;V^*)}
  &\leq \frac{\alpha_0+\Lambda_0}{d_{\mu,\beta}}
  \norm{R_{\mu,\nu}}_{L^2_\beta(0,\Tend;V^*)}.
  \label{eq:residual-dt}
\end{align}
Consequently,
\begin{align}
  &\norm{u_\mu-u_\nu}_{\W(0,\Tend)}
  +\norm{u_\mu-u_\nu}_{C([0,\Tend];H)} \notag\\
  &\quad \leq
  e^{\beta \Tend}
  \frac{1+\sqrt{\alpha_0}+\alpha_0+\Lambda_0}
  {d_{\mu,\beta}}
  \norm{R_{\mu,\nu}}_{L^2_\beta(0,\Tend;V^*)}. \label{eq:residual-full-unweighted}
\end{align}
\end{theorem}

\begin{proof}
By \Cref{lem:measure-convolution}, $R_{\mu,\nu}$ belongs to $L^2(0,\Tend;V^*)$ and hence, by \Cref{lem:weighted-equivalence}, to $L^2_\beta(0,\Tend;V^*)$.  We set $e:=u_\mu-u_\nu$. Because the data and the prescribed past coincide,
\begin{align*}
  \tilde u_\mu-\tilde u_\nu=\E_+e
  \quad\text{in }L^2(-\tau,\Tend;V),
\end{align*}
and $e(0)=0$ in $H$.  The bounded linearity of the history operators therefore gives
\begin{align*}
  \K_\mu\tilde u_\mu-\K_\nu\tilde u_\nu
  &=\K_\mu(\tilde u_\mu-\tilde u_\nu)
    +\K_{\mu-\nu}\tilde u_\nu 
  =\K_\mu^+e+R_{\mu,\nu}
  \quad\text{in }L^2(0,\Tend;V^*).
\end{align*}
Subtracting the two evolution equations yields
\begin{align}\label{eq:error-equation}
  \partial_te+A_0e+\K_\mu^+e=-R_{\mu,\nu}
  \quad\text{in }L^2(0,\Tend;V^*).
\end{align}
Applying \eqref{eq:parabolic-Lipschitz} to \eqref{eq:error-equation}, followed by \eqref{eq:weighted-causal}, gives
\begin{align*}
  \norm{e}_{L^2_\beta(0,\Tend;V)}
  \le \alpha_0^{-1}
  \left(
    \Lambda_1\kappa_\mu(\beta)
    \norm{e}_{L^2_\beta(0,\Tend;V)}
    +\norm{R_{\mu,\nu}}_{L^2_\beta(0,\Tend;V^*)}
  \right).
\end{align*}
Absorption proves \eqref{eq:residual-stability}.

To obtain the remaining estimates, write
\begin{align*}
  \partial_te+A_0e=F,
  \quad
  F:=-\K_\mu^+e-R_{\mu,\nu}.
\end{align*}
Using \eqref{eq:weighted-causal} and \eqref{eq:residual-stability},
\begin{align*}
  \norm{F}_{L^2_\beta(0,\Tend;V^*)}
  &\le \Lambda_1\kappa_\mu(\beta)
  \norm{e}_{L^2_\beta(0,\Tend;V)}
  +\norm{R_{\mu,\nu}}_{L^2_\beta(0,\Tend;V^*)}\\
  &\le
  \left(
    \frac{\Lambda_1\kappa_\mu(\beta)}{d_{\mu,\beta}}+1
  \right)
  \norm{R_{\mu,\nu}}_{L^2_\beta(0,\Tend;V^*)}\\
  &=\frac{\alpha_0}{d_{\mu,\beta}}
  \norm{R_{\mu,\nu}}_{L^2_\beta(0,\Tend;V^*)}.
\end{align*}
The weighted parabolic estimate \eqref{eq:weighted-parabolic-sup}, with zero initial value, now gives
\begin{align*}
  \sup_{0\le t\le \Tend}e^{-2\beta t}\norm{e(t)}_H^2
  \leq \alpha_0^{-1}
  \norm{F}_{L^2_\beta(0,\Tend;V^*)}^2,
\end{align*}
which proves \eqref{eq:residual-H}.  Finally, \eqref{eq:error-equation}, the boundedness of $A_0$, \eqref{eq:weighted-causal}, and \eqref{eq:residual-stability} yield
\begin{align*}
  \norm{\partial_te}_{L^2_\beta(0,\Tend;V^*)}
  &\le \Lambda_0\norm{e}_{L^2_\beta(0,\Tend;V)}
  +\Lambda_1\kappa_\mu(\beta)
  \norm{e}_{L^2_\beta(0,\Tend;V)}
  +\norm{R_{\mu,\nu}}_{L^2_\beta(0,\Tend;V^*)}\\
  &\le \frac{\alpha_0+\Lambda_0}{d_{\mu,\beta}}
  \norm{R_{\mu,\nu}}_{L^2_\beta(0,\Tend;V^*)},
\end{align*}
which proves \eqref{eq:residual-dt}.  Applying the norm equivalences of \Cref{lem:weighted-equivalence} separately to $e$ and $\partial_te$, and using
\begin{align*}
  \sup_{0\le t\le\Tend}\norm{e(t)}_H
  \leq e^{\beta\Tend}
  \sup_{0\le t\le\Tend}e^{-\beta t}\norm{e(t)}_H,
\end{align*}
proves \eqref{eq:residual-full-unweighted}.
\end{proof}

\begin{corollary}[Total-variation stability]\label[corollary]{cor:tv-stability}
Under the assumptions of \Cref{thm:residual-stability},
\begin{align}\label{eq:tv-stability}
  \norm{u_\mu-u_\nu}_{L^2_\beta(0,\Tend;V)}
  \leq
  \frac{\Lambda_1\norm{\mu-\nu}_{\M([0,\tau])}}
  {\alpha_0-\Lambda_1\kappa_\mu(\beta)}
  \norm{\tilde u_\nu}_{L^2(-\tau,\Tend;V)}.
\end{align}
Furthermore, let $\mu_n,\mu\in\M^0([0,\tau])$ satisfy
\begin{align*}
  \norm{\mu_n-\mu}_{\M([0,\tau])}\longrightarrow0,
\end{align*}
and let $u_{\mu_n}$ and $u_\mu$ be the solutions with common data.  Then,
\begin{align*}
  u_{\mu_n}\longrightarrow u_\mu
  \quad\text{in }\W(0,\Tend)\cap \mathcal{C}([0,\Tend];H).
\end{align*}
\end{corollary}

\begin{proof}
The convolution estimate \eqref{eq:measure-convolution-bound}, applied with $\sigma=\mu-\nu$, gives
\begin{align*}
  \norm{R_{\mu,\nu}}_{L^2_\beta(0,\Tend;V^*)}
  \leq
  \norm{R_{\mu,\nu}}_{L^2(0,\Tend;V^*)}
  \leq
  \Lambda_1\norm{\mu-\nu}_{\M([0,\tau])}
  \norm{\tilde u_\nu}_{L^2(-\tau,\Tend;V)}.
\end{align*}
Combining this bound with \eqref{eq:residual-stability} proves \eqref{eq:tv-stability}. For the sequential assertion, choose $\beta\ge0$ so that $d_{\mu,\beta}>0$.  Because
\begin{align*}
  \kappa_\sigma(\beta)
  =
  \norm{e^{-\beta(\cdot)}\sigma}_{\M([0,\tau])},
  \quad
  \sigma\in\M([0,\tau]),
\end{align*}
the reverse triangle inequality gives
\begin{align}
  \abs{\kappa_{\mu_n}(\beta)-\kappa_\mu(\beta)}
  &\leq
  \norm{e^{-\beta(\cdot)}(\mu_n-\mu)}_{\M([0,\tau])}
  \leq
  \norm{\mu_n-\mu}_{\M([0,\tau])}. \label{eq:kappa-tv-continuity}
\end{align}
Consequently,
\begin{align*}
  d_{\mu_n,\beta}
  \geq
  d_{\mu,\beta}
  -\Lambda_1\norm{\mu_n-\mu}_{\M([0,\tau])}
  \geq \frac12 d_{\mu,\beta}
\end{align*}
for all sufficiently large $n$. Applying \eqref{eq:residual-full-unweighted} with left-hand kernel $\mu_n$ and reference kernel $\mu$, and using $d_{\mu_n,\beta}\ge d_{\mu,\beta}/2$, we have
\begin{align*}
  &\norm{u_{\mu_n}-u_\mu}_{\W(0,\Tend)}
  +\norm{u_{\mu_n}-u_\mu}_{\mathcal C([0,\Tend];H)}
  \\
  &\quad \leq
  2e^{\beta\Tend}
  \frac{1+\sqrt{\alpha_0}+\alpha_0+\Lambda_0}
       {d_{\mu,\beta}}
  \norm{R_{\mu_n,\mu}}_{L^2_\beta(0,\Tend;V^*)}
  \longrightarrow0.
\end{align*}
This proves the asserted convergence.
\end{proof}

\begin{corollary}[Local Lipschitz continuity in total variation]
\label[corollary]{cor:local-tv-lipschitz}
Fix $\mu\in\M^0([0,\tau])$ and choose $\beta\geq0$ such that
$d_{\mu,\beta}>0$.  Let $r>0$ satisfy
\begin{align*}
  \Lambda_1 r\leq \frac12 d_{\mu,\beta}.
\end{align*}
Then, for every $\nu\in\M^0([0,\tau])$ with
$\norm{\nu-\mu}_{\M([0,\tau])}\leq r$, the corresponding solutions with
common data satisfy
\begin{align}\label{eq:local-tv-lipschitz}
  &\norm{u_\nu-u_\mu}_{\W(0,\Tend)}
  +\norm{u_\nu-u_\mu}_{\mathcal C([0,\Tend];H)} \notag\\
  &\quad\leq
  2e^{\beta\Tend}
  \frac{1+\sqrt{\alpha_0}+\alpha_0+\Lambda_0}{d_{\mu,\beta}}
  \Lambda_1\norm{\tilde u_\mu}_{L^2(-\tau,\Tend;V)}
  \norm{\nu-\mu}_{\M([0,\tau])}.
\end{align}
Thus, the solution map is locally Lipschitz from total variation into
$\W(0,\Tend)\cap\mathcal C([0,\Tend];H)$.
\end{corollary}

\begin{proof}
The reverse triangle inequality for the weighted variation masses gives
\begin{align*}
  d_{\nu,\beta}
  &\geq d_{\mu,\beta}
  -\Lambda_1\norm{\nu-\mu}_{\M([0,\tau])}
  \geq \frac12 d_{\mu,\beta}.
\end{align*}
Apply \eqref{eq:residual-full-unweighted} with left-hand kernel $\nu$
and reference kernel $\mu$.  The convolution estimate
\eqref{eq:measure-convolution-bound} yields
\begin{align*}
  \norm{\K_{\nu-\mu}\tilde u_\mu}_{L^2_\beta(0,\Tend;V^*)}
  \leq
  \Lambda_1\norm{\nu-\mu}_{\M([0,\tau])}
  \norm{\tilde u_\mu}_{L^2(-\tau,\Tend;V)}.
\end{align*}
Combining the preceding two estimates proves
\eqref{eq:local-tv-lipschitz}.
\end{proof}

\subsection{Narrow convergence}
We pass from total-variation convergence to narrow convergence.  Recall that $\mu_n\rightharpoonup\mu$ narrowly in $\M_+([0,\tau])$ if
\begin{align*}
  \int_{[0,\tau]}\varphi\,\dd\mu_n
  \longrightarrow
  \int_{[0,\tau]}\varphi\,\dd\mu
  \quad\text{for any }\varphi\in \mathcal{C}([0,\tau]).
\end{align*}
Because the interval is compact, narrow convergence also implies convergence, and therefore uniform boundedness, of the total masses.

\begin{lemma}[Banach-valued narrow convergence]\label[lemma]{lem:banach-narrow}
Let $Y$ be a Banach space, let $F\in \mathcal{C}([0,\tau];Y)$, and let $\mu_n\rightharpoonup\mu$ narrowly in $\M_+([0,\tau])$.  Then,
\begin{align}\label{eq:banach-narrow}
  \int_{[0,\tau]}F(s)\,\dd\mu_n(s)
  \longrightarrow
  \int_{[0,\tau]}F(s)\,\dd\mu(s)
  \quad\text{in }Y.
\end{align}
\end{lemma}

\begin{proof}
All integrals below are Bochner integrals.  Because $[0,\tau]$ is compact and $F$ is norm-continuous, $F$ is uniformly continuous, and its range is compact in $Y$.

Let $\varepsilon>0$.  Choose a finite relatively open cover $\{U_j\}_{j=1}^N$ of $[0,\tau]$ and points $s_j\in U_j$ such that
\begin{align*}
  \norm{F(s)-F(s_j)}_Y<\varepsilon,
  \quad s\in U_j.
\end{align*}
Let $\{\chi_j\}_{j=1}^N$ be a continuous partition of unity subordinate to this cover, and we define
\begin{align*}
  F_N(s):=\sum_{j=1}^N\chi_j(s)F(s_j).
\end{align*}
Then,
\begin{align*}
  \norm{F(s)-F_N(s)}_Y
  &\le
  \sum_{j=1}^N\chi_j(s)
  \norm{F(s)-F(s_j)}_Y
  <\varepsilon
\end{align*}
for any $s\in[0,\tau]$. By scalar narrow convergence,
\begin{align*}
  \int_{[0,\tau]}\chi_j\,\dd\mu_n
  \longrightarrow
  \int_{[0,\tau]}\chi_j\,\dd\mu,
  \quad j=1,\dots,N.
\end{align*}
Therefore,
\begin{align*}
  \int_{[0,\tau]}F_N\,\dd\mu_n
  \longrightarrow
  \int_{[0,\tau]}F_N\,\dd\mu
  \quad\text{in }Y.
\end{align*}
Furthermore, testing narrow convergence with the constant function $1$ shows that
\begin{align*}
  \mu_n([0,\tau])\longrightarrow\mu([0,\tau]),
\end{align*}
and therefore
\begin{align*}
  M_*:=\sup_n\mu_n([0,\tau])<\infty.
\end{align*}
Consequently,
\begin{align*}
  \norm{
    \int_{[0,\tau]} F\,\dd\mu_n-\int_{[0,\tau]} F\,\dd\mu
  }_Y
  &\leq
  \norm{
    \int_{[0,\tau]} F_N\,\dd\mu_n-\int_{[0,\tau]} F_N\,\dd\mu
  }_Y\\
  &\quad
  +\varepsilon\bigl(M_*+\mu([0,\tau])\bigr).
\end{align*}
Taking the upper limit as $n\to\infty$ gives
\begin{align*}
  \limsup_{n\to\infty}
  \norm{
    \int_{[0,\tau]} F\,\dd\mu_n-\int_{[0,\tau]} F\,\dd\mu
  }_Y
  \leq
  \varepsilon\bigl(M_*+\mu([0,\tau])\bigr).
\end{align*}
Because $\varepsilon>0$ is arbitrary, \eqref{eq:banach-narrow} follows.
\end{proof}

Although $G$ need not be continuous in time, the strong continuity of translations in $L^2(\R;V)$ turns narrow convergence of non-negative kernels into strong convergence of the corresponding history terms along the fixed trajectory $G$.

\begin{proposition}[Narrow convergence of history operators along a
fixed trajectory]\label[proposition]{prop:operator-narrow}
Let $\mu_n,\mu\in\M_+([0,\tau])$ with $\mu_n\rightharpoonup\mu$ narrowly, and let $G\in L^2(\R;V)$.  We define
\begin{align*}
  R_n
  :=
  \K_{\mu_n-\mu}G\big|_{(0,\Tend)}
  \in L^2(0,\Tend;V^*).
\end{align*}
Then,
\begin{align}\label{eq:operator-narrow}
  R_n\longrightarrow0
  \quad\text{in }L^2(0,\Tend;V^*).
\end{align}
Consequently, the same convergence holds in $L^2_\beta(0,\Tend;V^*)$ for any fixed $\beta\ge0$.
\end{proposition}

\begin{proof}
We set
\begin{align*}
  Y:=L^2(0,\Tend;V^*),
  \quad
  F(s):=A_1G(\cdot-s)\big|_{(0,\Tend)}.
\end{align*}
Translation continuity in $L^2(\R;V)$ and the boundedness of $A_1$ give
\begin{align*}
  \norm{F(s)-F(r)}_Y
  &\leq
  \Lambda_1
  \norm{G(\cdot-s)-G(\cdot-r)}_{L^2(\R;V)}
  \longrightarrow0
\end{align*}
as $s\to r$.  Therefore, $F\in\mathcal{C}([0,\tau];Y)$. We identify the $Y$-valued Bochner integral of $F$ with the history operator.  Let $\sigma\in\M([0,\tau])$ and $\Phi\in L^2(0,\Tend;V)$.  We have
\begin{align*}
  &\int_{[0,\tau]}\int_0^\Tend
  \abs{\pair{A_1G(t-s)}{\Phi(t)}}\,\dd t\,\dd|\sigma|(s)\\
  &\quad \leq
  \Lambda_1\norm{\sigma}_{\M([0,\tau])}
  \norm{G}_{L^2(\R;V)}
  \norm{\Phi}_{L^2(0,\Tend;V)}.
\end{align*}
Scalar Fubini's theorem therefore gives
\begin{align*}
  &\int_0^\Tend
  \pair{
    \left(\int_{[0,\tau]}F(s)\,\dd\sigma(s)\right)(t)
  }{\Phi(t)}\,\dd t \\
  &\quad =
  \int_{[0,\tau]}
  \int_0^\Tend
  \pair{A_1G(t-s)}{\Phi(t)}
  \,\dd t\,\dd\sigma(s)
  =
  \int_0^\Tend
  \pair{(\K_\sigma G)(t)}{\Phi(t)}
  \,\dd t.
\end{align*}
Since $V$ is Hilbert, the integrated duality pairing identifies
$L^2(0,\Tend;V)$ canonically with $Y^*$.  Because the preceding identity
holds for every $\Phi\in L^2(0,\Tend;V)$,
\begin{align}\label{eq:Y-integral-history-identification}
  \int_{[0,\tau]}F(s)\,\dd\sigma(s)
  =
  \K_\sigma G\big|_{(0,\Tend)}
  \quad\text{in }Y.
\end{align}
Applying \Cref{lem:banach-narrow} to $\mu_n$ and $\mu$, and then using \eqref{eq:Y-integral-history-identification}, yields
\begin{align*}
  R_n
  =
  \int_{[0,\tau]}F(s)\,\dd\mu_n(s)
  -
  \int_{[0,\tau]}F(s)\,\dd\mu(s)
  \longrightarrow0
  \quad\text{in }Y.
\end{align*}
This proves \eqref{eq:operator-narrow}.  Finally,
\begin{align*}
  \norm{R_n}_{L^2_\beta(0,\Tend;V^*)}
  \leq
  \norm{R_n}_{L^2(0,\Tend;V^*)}
\end{align*}
for any fixed $\beta\ge0$, which proves the weighted convergence.
\end{proof}

The following lemma shows that narrow convergence to a kernel without an atom at the origin yields a common positive weighted resolvent margin along the tail of the sequence.


\begin{lemma}[A common exponential weight under narrow convergence]
\label[lemma]{lem:common-weight}
Let $\mu_n,\mu\in\M_+([0,\tau])$ and assume that $\mu_n\rightharpoonup\mu$ narrowly with $\mu(\{0\})=0$. For any prescribed margin $d_*\in(0,\alpha_0)$, there exist $\beta\ge0$ and $N\in\mathbb N$ such that
\begin{align}\label{eq:common-weight}
\alpha_0-\Lambda_1\kappa_{\mu_n}(\beta)\ge d_*, \quad n\ge N,
\end{align}
and the same inequality holds with $\mu_n$ replaced by $\mu$.
\end{lemma}

\begin{proof}
We set $\eta:=\alpha_0-d_*>0$. Because $\mu(\{0\})=0$, \Cref{lem:kappa-properties} allows us to choose $\beta\ge0$ so that
\begin{align*}
  \Lambda_1\kappa_\mu(\beta)<\frac{\eta}{2}.
\end{align*}
For non-negative measures,
\begin{align*}
  \kappa_{\mu_n}(\beta)
  =
  \int_{[0,\tau]}e^{-\beta s}\,\dd\mu_n(s).
\end{align*}
Because $s\mapsto e^{-\beta s}$ is continuous on $[0,\tau]$, narrow convergence gives
\begin{align*}
  \kappa_{\mu_n}(\beta)\longrightarrow\kappa_\mu(\beta).
\end{align*}
Therefore, there exists $N\in\mathbb N$ such that
\begin{align*}
  \Lambda_1
  \abs{\kappa_{\mu_n}(\beta)-\kappa_\mu(\beta)}
  <\frac{\eta}{2},
  \quad n\ge N.
\end{align*}
Therefore,
\begin{align*}
  \Lambda_1\kappa_{\mu_n}(\beta)
  <
  \Lambda_1\kappa_\mu(\beta)+\frac{\eta}{2}
  <\eta,
  \quad  n\ge N,
\end{align*}
which proves \eqref{eq:common-weight}.  For the limit measure,
\begin{align*}
  \alpha_0-\Lambda_1\kappa_\mu(\beta)
  >
  \alpha_0-\frac{\eta}{2}
  =
  \frac{\alpha_0+d_*}{2}
  >d_*.
\end{align*}
\end{proof}

\begin{remark}[Where positivity enters the common-weight argument]
For signed measures the contraction quantity contains the variation measure $|\mu_n|$, whereas weak-star convergence of $\mu_n$ does not in general imply convergence of $|\mu_n|$.  The preceding argument is therefore intentionally stated for non-negative measures.  Total-variation convergence of signed measures still yields a common weight by \eqref{eq:kappa-tv-continuity}.
\end{remark}

The preceding common-weight and fixed-trajectory results combine to yield strong solution stability under narrow convergence.

\begin{theorem}[Narrow kernel stability]\label[theorem]{thm:narrow-stability}
Let $\mu_n,\mu\in\M_+^0([0,\tau])$ and assume
\begin{align*}
  \mu_n\rightharpoonup\mu
\end{align*}
narrowly on $[0,\tau]$. Let $u_n:=u_{\mu_n}$ and $u:=u_\mu$ be the corresponding solutions with common data.  Then,
\begin{align}\label{eq:narrow-solution-convergence}
  u_n\longrightarrow u
  \quad\text{in }\W(0,\Tend)\cap \mathcal{C}([0,\Tend];H).
\end{align}
\end{theorem}

\begin{proof}
Let $d_*:=\alpha_0/2$.  By \Cref{lem:common-weight}, there exist a common $\beta\ge0$ and $N\in\mathbb N$ such that
\begin{align}\label{eq:uniform-contraction}
  d_{\mu_n,\beta}\ge d_*,
  \quad n\ge N.
\end{align}
We extend the joined trajectory $\tilde u$ by zero outside $(-\tau,\Tend)$, and denote the resulting element of $L^2(\R;V)$ by $G$. Because
\begin{align*}
  -\tau<t-s<\Tend,
  \quad
   0<t<\Tend,\ 0 \leq s \leq\tau,
\end{align*}
the history terms generated by $G$ and $\tilde u$ coincide on $(0,\Tend)$. Therefore, \Cref{prop:operator-narrow} gives
\begin{align}\label{eq:residual-narrow-zero}
  \norm{\K_{\mu_n-\mu}\tilde u}
       {L^2_\beta(0,\Tend;V^*)}
  \longrightarrow0.
\end{align}
We apply \eqref{eq:residual-full-unweighted} with left-hand kernel $\mu_n$, reference kernel $\mu$, and the common weight $\beta$. Using \eqref{eq:uniform-contraction}, we obtain
\begin{align*}
  &\norm{u_n-u}_{\W(0,\Tend)}
  +\norm{u_n-u}_{\mathcal C([0,\Tend];H)}
  \\
  &\quad \leq
  e^{\beta\Tend}
  \frac{1+\sqrt{\alpha_0}+\alpha_0+\Lambda_0}{d_*}
  \norm{\K_{\mu_n-\mu}\tilde u}
       {L^2_\beta(0,\Tend;V^*)}.
\end{align*}
The right-hand side tends to zero by \eqref{eq:residual-narrow-zero}, which proves \eqref{eq:narrow-solution-convergence}.
\end{proof}

\begin{remark}[Why an atom at zero is excluded in \Cref{thm:narrow-stability}]
If a kernel sequence concentrates at zero, then $\kappa_{\mu_n}(\beta)$ need not become small uniformly in $n$ for any fixed $\beta$.  The limiting measure may contain an atom at the origin, in which case the limit is no longer a genuinely retarded problem: the atom changes the instantaneous form.  Such a limit must be formulated after splitting off the atom as in \Cref{rem:atom-zero}.  The present theorem covers narrow limits whose limiting measure has no atom at zero, including concentration at a positive delay.
\end{remark}

\begin{remark}[Why total variation is too strong for concentration]\label[remark]{rem:tv-concentration}
Let $\tau_d\in(0,\tau]$, and let $\mu_\varepsilon$ be absolutely continuous with respect to Lebesgue measure, non-negative, and of mass $m>0$.  Then, $\mu_\varepsilon$ and $m\delta_{\tau_d}$ are mutually singular, and hence
\begin{align*}
  \norm{\mu_\varepsilon-m\delta_{\tau_d}}_{\M([0,\tau])}
  =\mu_\varepsilon([0,\tau])+m
  =2m.
\end{align*}
Thus, total-variation convergence cannot detect distributed-to-discrete concentration.  Narrow convergence detects it qualitatively; quantitative transport estimates are developed in a companion manuscript.
\end{remark}

\begin{corollary}[Distributed-to-discrete convergence]
Let $m\geq0$ and $\tau_d\in(0,\tau]$.  Let $\mu_n\in\M_+([0,\tau])$ be absolutely continuous with respect to
Lebesgue measure and assume that
\begin{align*}
  \mu_n\rightharpoonup m\delta_{\tau_d}
  \quad\text{narrowly as }n\to\infty.
\end{align*}
Then, the corresponding solutions with common data satisfy
\begin{align*}
  u_{\mu_n}\longrightarrow u_{m\delta_{\tau_d}}
  \quad\text{in }
  \W(0,\Tend)\cap\mathcal{C}([0,\Tend];H).
\end{align*}
\end{corollary}

\begin{proof}
Absolute continuity gives $\mu_n(\{0\})=0$ for every $n$, while $\tau_d>0$ gives $(m\delta_{\tau_d})(\{0\})=0$. Therefore, $\mu_n,m\delta_{\tau_d}\in\M_+^0([0,\tau])$, and the assertion follows from \Cref{thm:narrow-stability}.
\end{proof}

\section{Concluding remarks}
We have established a finite-time theory for coercive evolution equations with measure-valued, form-valued delays.  For every finite signed retarded kernel, the weighted variation mass makes the unknown causal history a strict perturbation for a suitable kernel-dependent exponential weight.  The prescribed past remains an external forcing controlled by the full variation norm, and no positivity condition is imposed on the delayed form.

The residual principle separates the evolution estimate from the topology used to compare kernels.  It gives an asymmetric total-variation Lipschitz estimate, and hence local Lipschitz continuity, for signed measures, together with strong stability under narrow convergence for non-negative retarded measures whose limit has no atom at the origin.  The delayed-diffusion realisation shows that the theory applies when the delay contains principal spatial derivatives and cannot be reduced to a pivot-space delay equation.  As a direct consequence of narrow stability, distributed kernels converging weakly to an atom at a fixed positive lag generate strongly convergent solutions.

The narrow kernel-stability result also provides a consistency mechanism for numerical approximations in which a distributed delay law is replaced by a finite atomic quadrature.

Quantitative transport rates, regularity criteria for joined trajectories, and the singular boundary at zero delay are treated separately.  The present results retain a coercive instantaneous form; positive-type memory and graph-space formulations for degenerate instantaneous diffusion are studied in \cite{IshizakaCoercivity2026,IshizakaGraph2026}.  The kernel-stability estimates obtained here are, in addition, the analytic basis for identifying a memory law from observations of the state and for feedback through distributed or discrete delay; these control-theoretic questions are left to future work.

\section*{Funding}
The author declares that no funds, grants, or other support were received during the preparation of this manuscript.

\section*{Data availability}
No datasets were generated or analysed during the current study.

\section*{Competing interests}
The author declares no competing interests.

\end{document}